\documentclass[11pt, reqno]{amsart}

\usepackage{amsfonts,amsmath,amsthm}
\usepackage{amssymb,epsfig}
\usepackage{cases, verbatim}
\usepackage[usenames]{color}
\usepackage{esint}
\usepackage{setspace}

\usepackage{enumerate} 

\usepackage[shortlabels]{enumitem}
\usepackage{mathscinet}
\usepackage{mathtools}
\usepackage{color} 
\definecolor{vert}{rgb}{0,0.6,0}

\usepackage[sort, numbers]{natbib}
\theoremstyle{plain}
\newtheorem{thm}{Theorem}

\newtheorem{defn}[thm]{Definition}

\newtheorem{lem}[thm]{Lemma}
\newtheorem{cor}[thm]{Corollary}
\newtheorem{prop}[thm]{Proposition}
\theoremstyle{remark}
\newtheorem{rem}{Remark}[section]

\newcommand{\C}{\operatorname{cap}_Q}

\newcommand{\N}{\mathbb{N}}
\newcommand{\R}{\mathbb{R}}

\newcommand{\al}{\alpha}

\newcommand{\eps}{\epsilon}

\newcommand{\Om}{\Omega}

\newcommand{\sub}{\subset}
\newcommand{\ra}{\rightarrow}

\newcommand{\loc}{{\rm loc}\,}

\newcommand{\diam}{\operatorname{diam}}
\newcommand{\dist}{\operatorname{dist}}

\renewcommand{\:}{\colon}
\newcommand{\capa}{\operatorname{cap}}

\newcommand{\Capab}{\C(\{g\ge \beta\}, \{g>\alpha\})}

\newcommand{\Capabp}{\C(\{g\ge \beta'\}, \{g>\alpha'\})}

\numberwithin{thm}{section} 
\numberwithin{equation}{section}

\begin{document}
\title[\ Green functions in metric measure spaces]
{  Green functions in  metric measure spaces}
\author{Mario Bonk}
\address{Mario Bonk, Department of Mathematics, University of California, Los Angeles,
CA 90095, USA\\
{\tt mbonk@math.ucla.edu}}
\author{Luca Capogna}
\address{Luca Capogna, Department of Mathematics and Statistics, Smith College, Northampton, MA 01063, USA \\
{\tt  lcapogna@smith.edu}}

\author{Xiaodan Zhou}
\address{Xiaodan Zhou, Analysis on Metric Spaces Unit, Okinawa Institute of Science and Technology Graduate University, Okinawa 904-0495, Japan\\ 
{\tt xiaodan.zhou@oist.jp}}

\date{\today}

\thanks{M.B.\ was partially supported by NSF award DMS-1808856}
\thanks{L.C.\ was partially supported by NSF award  DMS-1955992}
\thanks{X.Z.\ was partially supported by JSPS award 22K13947}

\begin{abstract} 
We study existence and uniqueness of Green functions for the Cheeger $Q$-Laplacian in metric measure spaces that are Ahlfors $Q$-regular and support a $Q$-Poincar\'e inequality with $Q>1$.  We prove uniqueness of Green functions both in the case of relatively compact domains, and in the global (unbounded) case. We also prove existence of global Green functions in unbounded spaces, complementing the existing results in  relatively compact domains proved recently in \cite{BBLe20}.
\end{abstract}

\maketitle

\section{Introduction}
In  Euclidean space $\R^n$, a result of Kichenassamy and Ver\'on \cite[Theorem 2.2]{KV86} shows that the $n$-Laplace operator
$$\mathcal{L}_n u:= \text{div}(|\nabla u|^{n-2} \nabla u) $$
admits a unique  global Green function, i.e., there is a unique, properly normalized singular solution which blows up to $+\infty$ at the origin and converges to $-\infty$ at infinity.  Their proof of uniqueness applies also to the $p$-Laplacian $\mathcal{L}_p$ in the range $1<p\le n$ and  is based on $C^{1,\alpha}$ estimates for  
$p$-harmonic functions. In the conformal case $p=n$, the argument  was later simplified and extended to the Riemannian and Carnot group setting in  \cite{Hol90, BHT02} thus establishing uniqueness of global Green functions  in these geometries.  

The purpose of this paper is to extend this uniqueness result to the setting of 
 complete metric spaces $(X,d, \mu)$ equipped with an Ahlfors regular Borel measure $\mu$, and a Poincar\'e inequality.  In the following we describe this structure more in detail and make sense of the notion of $p$-Laplacian in this generality. 
 
 Let 
 $u\:X\to \R$ be a function on $X$. We say that a Borel measurable 
 function $g\:X\to [0,\infty]$ is an {\em upper gradient of $u$} if 
 $$|u(x)-u(y)|\le \int_\gamma g\, ds,$$ for all $x,y\in X$ and 
  all rectifiable paths 
 $\gamma$ in $X$ joining $x$ and $y$. 
For $p\ge 1$ the  $L^p$-integrable  functions $u$ on $X$ that have 
  an $L^p$-integrable upper gradient form the generalized {\em Sobolev} space (or {\em Newtonian} space) $N^{1,p}(X,d,\mu)$ (see \cite{Sha00,Sha01}). 
  
  Throughout this paper we will assume that  $(X,d,\mu)$ is a  {\em PI-space}, 
  that is, a doubling metric measure spaces supporting a $p$-Poincar\'e inequality for some $p\ge 1$ (see \eqref{poincareinequality}).  By a  breakthrough result of  Cheeger \cite{Che99}, the space $(X,d,\mu)$ can then
   be equipped with a  {\em measurable differentiable structure}. 
 Accordingly,  for every $u\in N^{1,p}(X,d,\mu)$ one can define a differential $Du$, whose pointwise Euclidean norm $|Du|$ is comparable to the {\em minimal upper gradient}  $g_u$ of  $u$ (see Section~\ref{Newtonian spaces} below for details).

 Let $\Om\subset X$ be a {\em region}, that is, an open and connected set. Then for each $1<p<\infty$ one can define two notions of $p$-harmonic functions on $\Omega$ (for the definition of the Sobolev spaces used here, see Section~\ref{Newtonian spaces}):

 \smallskip
(i) {\em $p$-energy minimizers} are functions 
 $u\in N^{1,p}_{loc}(\Om)$
 such that
for all subsets $U \subset \subset \Omega$
 and for all functions $\phi\in N^{1,p}_0(U)$,
\begin{equation}\label{pmin}
\int_U g^p_u\le \int_U g^p_{u+\phi}.
\end{equation}
Here and in similar expressions, integration is with respect to the underlying doubling measure $\mu$.   If the ambient space admits a smooth differentiable structure, then one can derive a (second order) Euler-Lagrange equation  that characterizes minimizers of  $p$-energy, but in the general case such a PDE is not 
available. Accordingly, in this paper we will use the term {\em $p-$energy minimizer}   in the sense of \eqref{pmin} to emphasize the non-availability of a PDE.

\smallskip 
(ii)   A function $u\in N^{1,p}_{loc}(\Om)$ is {\em (Cheeger)  
$p$-harmonic} in $\Om$ for a choice of a fixed measurable differentiable structure on the underlying space $(X,d,\mu)$   we have   
\begin{equation}\label{pde-eq}
\int_\Om |Du|^{p-2}Du\cdot D\phi =0\end{equation}
for all  functions $\phi\in N^{1,p}_0(\Om)$. 

\smallskip
Equation \eqref{pde-eq}  is  a distributional formulation of an underlying PDE,  analogue to the Euclidean $p-$Laplacian; it is linear in case $p=2$, but non-linear if $p\ne 2$. 
While  the notion of $p$-energy minimizers  in the sense of \eqref{pmin} only  depends on  the metric  and the measure of the underlying space, $p$-harmonicity depends on the choice 
of a measurable differentiable structure (more precisely, on the choice of 
fiberwise inner product on the measurable cotangent bundle; see Section~\ref{differentiability}).  In this paper, we will assume that a measurable differentiable structure on the underlying space has been fixed  and 
$p$-harmonicity is with respect to this  fixed structure.

One can define a notion of fundamental solution, and Green function for the Cheeger $p$-Laplacian (see Definitions \ref{bd def} and \ref{global-Green} below). As in the Euclidean and Riemannian setting, such functions are the cornerstone of a non-linear potential theory, which in turn has a broad range of applications, from  quasiconformal geometry to boundary value problems; see for instance \cite{Hol90, HK98, CHSC01, CHSC02, BMS01,HST01}.

Our first main result is about the  existence and uniqueness of non-negative Green functions of the Cheeger $p$-Laplacian in relatively compact subregions  
of  PI-spaces. 


\begin{thm}\label{1.1}
Let $(X, d, \mu)$ be a complete doubling metric measure space. Let ${p}>1$ and suppose $(X,d,\mu)$ supports a ${p}$-Poincar\'e inequality. We  fix a measurable differentiable structure on $(X, d, \mu)$. 

 Let $\Omega$ be a relatively compact region in $X$,  and $x_0\in \Omega$. Assume that every point $x\in \partial \Omega$ is a regular point for the Dirichlet problem for the $p$-Laplacian.   Then the  following statements hold:

\begin{enumerate}[label=\text{(\roman*)},font=\normalfont,leftmargin=*]

\item If the Sobolev $p$-capacity of the complement of $\Om$ is positive, i.e., if  $C_p(X\setminus\Omega)>0$ (see \eqref{definition-capacity} for the definition), then there exists a Green  function on $\Omega$ for the Cheeger 
$p$-Laplacian with singularity at $x_0$. 
\smallskip
\item  Assume in addition that $(X,d,\mu)$ is $Q$-Ahlfors regular and {$p=Q>1$}. Let  $u, v$ be  two Green functions on $\Omega$ for the Cheeger  $Q$-Laplacian with singularity at $x_0$. If there exists $y_0\in \Omega\setminus \{x_0\}$ such that $u(y_0)=v(y_0)$, then $u=v$ on $\Omega\setminus\{x_0\}$.
\end{enumerate}
\end{thm}
The existence part of Theorem \ref{1.1} was recently established  
 in \cite{BBLe20};  it is an analog for the Cheeger $p$-Laplacian of the existence of $p$-energy minimizers  in the sense of \eqref{pmin} proved by  Holopainen and Shanmugalingam in \cite{HS02}. In the special case of Carnot groups, the existence  of Green functions for the $p$-Laplacian was proved  in \cite{DG00} for $1<p<Q$, where $Q$ is the homogeneous dimension. Our contribution is the uniqueness statement (ii), which  rests on a more involved argument, and extends  an earlier result of Holopainen \cite{Hol90} from the Riemannian setting to the class of Ahlfors regular PI-spaces.

\begin{rem} In our proof of  uniqueness of  Green functions in Theorem \ref{1.1} we never use  the PDE \eqref{pde-eq} directly, and in fact the same arguments imply the uniqueness of  $Q-$energy minimizing  Green functions   in the sense of \eqref{pmin}, as studied for instance in \cite{HS02,KS01, BBLe17, BBLe20, BBLe23}, {\it without having to rely on a Cheeger structure}. We note that unlike in the Euclidean case \cite[Theorem 2.2]{KV86},  for $1<p<Q$ the uniqueness for $p-$energy  minimizers in the sense of \eqref{pmin} may fail in this more general setting.
In fact,  for   $1<p<n$, the recent work   \cite{BBEZ24} of one of us jointly with A. Bj\"orn, J. Bj\"orn, and Eriksson-Bique, establishes the existence of  an Ahlfors $n-$regular Finsler metric in $\mathbb{R}^n$ with   uncountably many distinct    $p-$energy minimizing  Green functions in the sense of \eqref{pmin}, in a disk. Uniqueness of the Cheeger $p-$energy minimizers in this setting may still hold, but the proof would require a strategy  different from our current approach.
\end{rem}
With the added assumption of Ahlfors $Q$-regularity, we also establish existence and uniqueness of unbounded global Green functions for the Cheeger $Q$-Laplacian. 

\begin{thm}\label{thm:ubddGreen} Let $(X, d, \mu)$ be an unbounded, complete Ahlfors $Q$-regular metric measure space that supports a $Q$-Poincar\'e inequality for $Q>1$. Then for each $x_0\in X$ the following statements are true: 
\begin{enumerate}[label=\text{(\roman*)},font=\normalfont,leftmargin=*]

\item  The space is $Q$-parabolic and  there exists a global $Q$-harmonic Green function for the $Q$-Laplacian unbounded on both sides, with singularity at $x_0$.

\item For each global $Q$-harmonic Green function $u$, there exists a constant $C\ge 1$ depending on the structure constants of $X$ and the value of $u$ on a fixed radius such that for all $x\in X\setminus \{x_0\}$, we have
 \[
 C^{-1}\log\bigg(\frac{1}{d(x,x_0)}\bigg)-C\le u(x)\le C\log\bigg(\frac{1}{d(x,x_0)}\bigg)+C.
 \]

\item   
 Let $u, v$ be two global $Q$-harmonic Green functions with singularity at $x_0$ in $X$. If there exists $y_0\in X\setminus\{x_0\}$ such that $u(y_0)=v(y_0)$, then $u=v$ on $X\setminus\{x_0\}$.
\end{enumerate}
 \end{thm}
 
Global Green functions are defined  in Definition \ref{global-Green}.  

In the special case of  Carnot groups,  the existence part of this theorem had been established earlier by Heinonen and Holopainen \cite{HH97}  and  the uniqueness part  by Balogh, Holopainen and Tyson \cite{BHT02}.

\begin{rem} As in the bounded domain case, in  our proof of existence and uniqueness of  global Green functions  we never use  the PDE \eqref{pde-eq} directly, and so it would appear as if one could use the same arguments to prove existence and uniqueness of global $Q-$energy minimizing  Green functions  in the sense of \eqref{pmin}.
However, while the proof of existence goes through without problems,  the uniqueness part of the proof of  Theorem \ref{thm:ubddGreen} relies in a crucial way  on the Clarkson inequalities in Lemma \ref{clarkson-lemma}, and such estimates are only known in the presence of an inner product structure on the target space of the functions involved. In our specific application of the Clarkson inequality, this translates into the need of having an inner product on the generalized cotangent space of the metric space, which is guaranteed by the Cheeger differentiability structure. 

While Theorem \ref{thm:ubddGreen} shows that the inner product structure hypothesis is sufficient to prove uniqueness,  in a sense this is also a necessary assumption. In fact,  one of the results in   \cite{BBEZ24} is a construction of an Ahlfors $n-$regular Finsler metric in $\mathbb{R}^n$ with   uncountably many    global Finsler $n-$energy minimizing  Green functions in the sense of \eqref{pmin} which have distinct minimal upper gradients. While these Finsler spaces are PI, and hence admit a Cheeger structure, Theorem \ref{thm:ubddGreen} cannot be applied to  the minimizers of the Finsler $n-$energy in the sense of \eqref{pmin} as these are only quasiminimizers of the Cheeger $n-$energy.
\end{rem}

 Thanks to the work of Danielli, Garofalo and Marola \cite{GM10, DGM10} and Bj\"orn,  Bj\"orn and Lehrb\"ack \cite{BBLe17, BBLe20, BBLe23}, one has estimates on the blow up rate of  singular solutions near the pole, generalizing earlier work of Serrin \cite{Ser65}  from the Euclidean setting to the PI-setting. In the same paper, Serrin proved also  that any singular solutions bounded from below is a multiple of a fundamental solution (i.e.,  when applying the operator, one obtains a multiple of Dirac's delta at the pole).  We note that this result  continues to hold in the PI setting. 

\begin{thm}\label{fundamental solution}
Let $(X, d, \mu)$ be a complete Ahlfors $Q$-regular metric measure space that supports a $Q$-Poincar\'e inequality for $Q>1$. Let $\Om$ be a relatively compact domain in $X$ and let $u$ be a $Q$-harmonic Green function in $\Omega$ with singularity at $x_0\in \Om$. There exists $K\in\R$ such that for all Lipschitz continuous function $\phi\in N^{1,Q}_0(\Omega)$ one has 
$$\int_\Om |Du|^{Q-2}Du\cdot D\phi  = K\phi(x_0).$$
\end{thm}

A similar result is proved by Bj\"{o}rn, Bj\"{o}rn and Lehrb\"{a}ck in \cite[Theorem 13.3]{BBLe20}.

This paper is organized in the following way. In Section 2, we give a brief review of some basic definitions and well-known results. In Section 3, we define a Green function on a relatively compact domain and show the existence and uniqueness under suitable conditions. In Section 4, we consider Green functions defined on  an unbounded space $X$. In Section 5 we provide a proof for Theorem \ref{fundamental solution}.

\subsection*{Acknowledgments} The first-named author would like to thank Bruce Kleiner for many discussions related to Green functions on metric measure spaces during  a visit to the Courant Institute in January 2017. 
 In particular,  a proof of Theorem~\ref{thm:ubddGreen} for $Q\ge 2$ emerged from these  discussions. 
  
 The authors are also  grateful to Anders Bj\"orn,  Jana Bj\"orn, Sylvester Eriksson-Bique and Nageswari Shanmugalingam for several conversations about the topics in this paper.

\section{Preliminaries}

Throughout this paper,  $(X,d,\mu)$ will be a space $X$ equipped with a metric $d$ and a Borel measure
$\mu$. We always assume that $X$ is complete and that  $0<\mu(B)<\infty$ for every  ball $B\subset X.$ 

We denote by 
$$B(x,r)=\{y\in X: d(x,y)<r\}$$ the open ball of radius $r>0$ centered at $x\in X$. If $B=B(x,r)$ is any (open) ball and $\lambda\ge 1$, then we use the notation $\lambda B=B(x, \lambda r)$. We write $\diam(A) $ for the diameter of a set $A\sub X$.

We say the measure $\mu$ or the space $(X,d,\mu)$  is {\em doubling} if there exists a constant $C\ge 1$ such that $\mu(2B)\le C\mu(B)$ for every open ball $B\subset X$. 
We say that  $(X,d,\mu)$ is {\em Ahlfors $Q$-regular} for $Q\ge 1$ if there is a  constant $C\ge 1$ such that
\[
C^{-1}r^Q\le\mu(B(x,r))\le Cr^Q
\]
for all $x\in X$ and $0<r\le{\rm diam}(X)$.

The space  $(X, d, \mu)$ is said to support a {\em $p$-Poincar\'e inequality}, $p \ge 1$,  if there exist constants $C>0$ and $ \lambda\ge1$ such that for every measurable function 
$u\: X\to \mathbb{R}$ and every upper gradient $g\:X \to [0,\infty]$ of $u$, the pair $(u,g)$ satisfies
\begin{equation}\label{poincareinequality}
\frac {1}{\mu(B)}\int_B |u-u_B|\, d\mu\le C\diam (B)\left( \frac {1}{\mu(\lambda B)}\int_{\lambda B}g^p\, d\mu\right)^{1/p}
\end{equation}
on every open ball $B\subset X$.  Here $u_B= \frac {1}{\mu(B)}\int_B u\, d\mu$.

 We say  $(X,d,\mu)$ is a {\em $p$-PI-space}, $p\ge 1$,  if it is doubling and  supports a $p$-Poincar\'e inequality. Finally, $(X,d,\mu)$ is a {\em PI-space} if it is a $p$-PI-space for some $p\ge 1$.

\subsection{Newtonian spaces.}\label{Newtonian spaces}
Let $u\: X\to \mathbb{R}$ be a function. A Borel function $g\:X\to [0,\infty]$ is said to be a {\em $p$-weak upper gradient}  of $u$ if
\[|u(\gamma(a))-u(\gamma(b))|\le \int_{\gamma}g\ ds\]
holds for all rectifiable paths  $\gamma\:[a,b]\to X$ outside a family  of paths  with vanishing $p$-modulus. See \cite{HK98, Che99, Sha00} for 
more details on upper gradients, including the definition of $p$-modulus of a path family. 

For $1\le p<\infty$ the space  $\widetilde{N}^{1,p}(X,d,\mu)$ consists of all $L^p$-integrable functions on $X$ for which there exists a $p$-integrable $p$-weak upper gradient. For each $u\in \widetilde{N}^{1,p}(X,d,\mu)$ we define
\[
\|u\|_{\widetilde{N}^{1,p}(X,d,\mu)}=\|u\|_{L^p(X)}+\inf_g\|g\|_{L^p(X)},
\]
where the infimum is taken over all $p$-weak upper gradients $g$ of $u$.
An equivalence relation in $\widetilde{N}^{1,p}(X,d,\mu)$ is defined as $u \sim v$ if $ \|u-v\|_{\widetilde{N}^{1,p}(X,d,\mu)}=0$. By definition the {\em Newtonian space} $N^{1,p}(X,d,\mu)$ is the quotient space $\widetilde{N}^{1,p}(X,d,\mu)/ \sim\ $. As usual in this and similar contexts, we think of the elements in $N^{1,p}(X,d,\mu)$  as equivalence classes, but rather 
as represented by each function in its class.  With this understanding, $N^{1,p}(X,d,\mu)$  is a Banach space equipped with the norm $\|u\|_{N^{1,p}}=\|u\|_{\widetilde{N}^{1,p}}.$ For each $u\in N^{1,p}(X,d,\mu)$ there exists an essentially unique {\em minimal $p$-weak upper gradient}  $g_u\in L^p$ of $u$ in the sense that if $g\in L^p$ is another $p$-weak upper gradient of $u$, then $g_u\le g$ a.e.\ on $X$.
 
The space  $N_{\loc}^{1,p}(\Omega)$ consists of all 
function $u\:\Omega \ra \R$  such that  if $u|_{\Om'} \in N^{1,p}(\Omega')$ for every open set $\Omega'\subset \subset\Omega$.

For $p\ge 1$ the {\em Sobolev $p$-capacity} of a set $E\subset X$ is defined as \begin{equation}\label{definition-capacity}
{C_p}(E)=\inf_u \| u \|^p_{N^{1,p}(X)},\end{equation}  where the infimum is taken over all $u\in N^{1,p}(X)$ such that $u\ge 1$ on $E$. A property for points $x\in X$  is said to hold {\em $p$-quasi-everywhere}, abbreviated 
{\em $p$-q.e.},  if it is true outside a set of $p$-capacity zero. 

For  $\Omega\subset X$ we denote by ${N}^{1,p}_0(\Omega)$ is the set  of all functions $u\in N^{1,p}(X)$ whose representative functions vanish $p$-quasi-everywhere in $X\setminus \Omega$.  Then $N_0^{1,p}(\Omega)$ is a Banach space equipped with the norm $\|u\|_{N^{1,p}_0(\Omega)}=\|u\|_{N^{1,p}(X)}$ \cite[Theorem 4.4]{Sha01}.

\subsection{Measurable differentiable structures}\label{cheeger}

\begin{defn}[Cheeger-Keith]\label{differentiability}
A {\em measurable differentiable structure} on $(X,d,\mu)$ is a countable collection of pairs $\{U_\al,x_\al\}_{\alpha\in I}$ called {\em coordinate patches} with the following properties: 

\begin{enumerate}[label=\text{(\roman*)},font=\normalfont,leftmargin=*]
\item  Each $U_\al$, $\alpha\in I$,  is a  measurable subset of $X$  with positive measure, and the union $\bigcup_{\alpha\in I}U_\alpha$ has full measure in $X$.

\smallskip
\item  Each $x_\al\: X\ra \R^{N(\alpha)}$, $\alpha\in I$,  is a Lipschitz map on $X$, where  $N(\al)\in \N$ bounded above independently of $\al\in I$.

\smallskip
\item For every Lipschitz function $f\:X\to \R$ and every $\alpha\in I$  there exists an $L^\infty$-map  $Df^\al\:X\to \R^{N(\al)}$ such that for $\mu$-a.e.\ $x\in U_\alpha$ we have 
$$\limsup_{z\to x} \frac{1}{d(x,z)}\bigg|  f(z) -  f(x)-Df^\al(x)\cdot (x_\al(z) -x_\al(x)) \bigg|=0.$$
\end{enumerate} 
\end{defn}

The following result by   Cheeger \cite{Che99} provides measurable differentiable structures for fairly general spaces. 

\begin{thm}[Cheeger]
If $(X,d,\mu)$ is a metric measure space that is doubling and supports a $ p$-Poincar\'e  inequality for some $p\ge 1$, then it admits a measurable differentiable structure with dimension bounds depending only on $p$, and the constants 
in the doubling condition and the Poincar\'e inequality. 
\end{thm}
Cheeger's result yields a finite-dimensional  vector bundle $T^* X$, and to each Lip\-schitz function $f\:X\to \R$ one can associate a $L^\infty$-section {$Df=\sum_\al Df^\al\chi_{U_\al}$} of this bundle. Cheeger proved that the  $Df$  has (Euclidean)  length $|
D f|$ comparable to the minimal upper gradient $g_f$, i.e. there exists a constant $C>0$ depending only on the doubling and Poincar\'e constants such that for all Lipschitz functions $f$, one has
\begin{equation}\label{comparable}
K^{-1} g_f \le |Df|\le K g_f.
\end{equation}

{More precisely, Cheeger showed that there exists a norm for $\|Du\|$ such that $\|Df\|=g_u$, but this norm is not necessarily generated by an inner product. If the norm is generated by an inner product, we have equality above. Otherwise, one can choose a norm $|\cdot|$ generated by the inner product and bi-Lipschitz equivalent to $\|\cdot \|$ to obtain \eqref{comparable}.} For alternative expositions and recent developments of the results \cite{Che99} see  \cite{Ke04, KM16,B15,B20,BL} and references therein.

 Note that the operator $D$ can be extended to functions in $N^{1,p}(X,d, \mu)$, and that the Cheeger gradients satisfy both the product rule and the chain rule. Assume $u\in N^{1,p}(X)$, $f$ is a bounded Lipschitz function on $X$ and $h\:\mathbb{R}\to\mathbb{R}$ is continuously differentiable with bounded derivative, then $uf$ and $h\circ u$ both belong to $N^{1,p}(X)$ and
\[
D(uf)=uDf+fDu
\]
\[
D(h\circ u)=(h' \circ u) \cdot Du.
\]

Next, we recall the notion of $p$-harmonicity related to the differentiable structure:
\begin{defn}\label{cpl}
 A function $u\in N^{1,p}_{loc}(\Om)$
is {\em (Cheeger)  $p$-harmonic} in a domain $\Om$ if 
for all subsets $U \subset \subset \Om$
 and for all functions $\phi\in N^{1,p}_0(U)$,
 $$\int_U |D u|^p  \, d\mu\le \int_U |D (u+\phi)|^p \, d\mu.$$
\end{defn}

It follows immediately that the definition above is equivalent to 
\begin{equation}\label{cpl1}\int_U |Du|^{p-2}(Du\cdot D\phi ) \, d\mu=0,\end{equation}
for $U, u,\phi$ as above.    In view of the comparability estimate \eqref{comparable}, one has that any 
(Cheeger) $p$-harmonic function $u\in N^{1,p}(\Om)$ is a  {\em  quasiminimizer} of the $p$-energy functional 
\begin{equation}\label{penergy}
f\to \int_{\Om} g_f^p \, d\mu,
\end{equation}  i.e.,  there exists a constant $K>0$ such that for all bounded open subsets $U$, with  $U\subset \bar{U}\subset \Om$, and for all $v\in N^{1,p}(U)$ with $u-v\in N^{1,p}_0(U)$ the inequality
$$\int_{U \cap \{u\neq v\}} g_u^p \, d\mu \le K \int_{U \cap \{u\neq v\}} g_v^p \, d\mu,$$
holds.

In view of the Harnack inequality proved by   Kinnunen and Shanmugalingam  in \cite{KS01} one has the following statement. See also \cite[Theorem 8.10]{BB11}.

\begin{prop} For any compact set $K\subset \Om$, there exists  $C>0$ depending only on $K$ and on the constants in the doubling measure and in the Poincar\'e inequality, such that for any $p$-harmonic function $u\ge 0$ in $\Om$, one has
$$\sup_{B(x,R)} u \le C \inf_{B(x,R)} u,$$
for all $x\in K$, and $R>0$ such that $B(x,100\lambda R)\subset \Om$, with $\lambda$ as in \eqref{poincareinequality}.
\end{prop}
As a corollary one obtains a local estimate of the H\"older seminorm $$[u]_{\al,B(x,R)} = \sup_{ 0<d(x,y)<R} \frac{|u(x)-u(y)|}{d(x,y)^\alpha}$$ of $p$-harmonic functions 
{\cite[Theorem 8.14]{BB11}}.
\begin{cor} Under the same hypotheses as in  the previous proposition, there exist $\al\in (0,1)$ and $M>0$ depending only on $K$ and on the constants in the doubling measure and in the Poincar\'e inequality, such that any $p$-harmonic function $u$ in $\Om$ 
{can be modified on a set of capacity zero to a continuous function $\tilde{u}$} and
$[\tilde{u}]_{\alpha, B(x,R)}\le M$, for all $x\in K$, and $R>0$ such that $B(x,6R)\subset \Om$.
\end{cor}

Through an argument similar to the one in \cite[Lemma 3.18]{HKM06}, one can prove that  $p$-harmonic functions and $p$-energy minimizers both satisfy a comparison principle (see e.g., \cite[Theorem 7.17]{Che99} or \cite[Theorem 6.4]{Sha01}).

\begin{lem}[Comparison principle]
Let $X$ be a PI space and let $u, v\in N^{1,p}(\Om)$ denote two $p$-harmonic functions in  
an open set $\Om\subset X$. If $u\le v$ $p$-quasi-everwhere in $X\setminus \Omega$, then $u\le v$ almost everywhere in $\Om$. 
\end{lem}
An immediate corollary is the following maximum (minimum) principle \cite[Corollary 6.6]{Sha01}, the interior value of a $p$-harmonic function in a domain $\Omega$ can not be larger or smaller than its boundary value. 
\begin{cor} 
Let $X$ be a PI space and let $u\in N^{1,p}(\Om)$ is 
$p$-harmonic in an open set $\Om\subset X$, then for all $\eta>0$, we have
\[
\sup_{\Omega\setminus \Omega_\eta} u=\sup_\Omega u\quad\quad{and}\quad\quad\inf_{\Omega\setminus \Omega_\eta}=\inf_\Omega u,
\]
where $\Omega_\eta=\{x\in \Omega\colon B(x,2\eta)\subset \Omega\}.$
\end{cor}



We will need the following result which is well known in the Euclidean setting (see \cite[Theorem 3.78]{HKM06}). 
\begin{prop} If $\{u_i\}$ is a sequence of $p$-harmonic functions in $\Om$, converging uniformly on compact subsets to a continuous function $u$,
then $u$ is also $p$-harmonic in $\Om$. The result continues to hold if one assumes $u\in N^{1,p}(\Om)$ instead of continuity. 
\end{prop} 
The proof is very similar to \cite[Theorem 1.2]{Sha03}, where the result is proved for minimizers of the $p$-energy \eqref{penergy}.

 
 Let $\Omega\subset X$ be a bounded open set, and $\phi\in C(\partial \Omega) \cap N^{1,p}(X)$. Solving the Dirichlet problem for the (Cheeger) $p$-Laplacian in $\Omega$, with boundary values $\phi$ consists in finding the unique $p$-harmonic function $u$ in $\Omega$ such that \begin{equation}\label{boundary} u-\phi \in N^{1,p}_0(\Omega).\end{equation} Uniqueness follows immediately from the comparison principle above. Existence has been studied by several authors (see for instance \cite{Sha01,Che99}). In \cite{BBS-3} it is proved that the (Perron) solution always exists and satisfies \eqref{boundary}. The boundary of $\Omega$ is said to be {\it regular} for the Dirichlet problem for the $p$-Laplacian, if the unique solution to the Dirichlet problem is continuous in $\overline {\Omega}$.


\subsection{Relative $p$-capacity}\label{capacity}

Let $\Omega\subset X$ be open and $K\subset \Omega$ be a bounded set. The {\em relative $p$-capacity} of $K$ with respect to $\Omega$ is defined as 
\[
\capa_p(K, \Omega)=\inf_{u} \int_{\Omega} {|Du|}^p \  ,
\]
where the infimum is taken over all functions $u\in N^{1,p}(X)$ such that $u=1$ in $K$ {$p$-quasi-everywhere} and $u=0$ in $X\setminus \Omega$ {$p$-quasi-everywhere}. If no such functions exist, we set $\capa_p(K, \Omega)=\infty.$ 
When $\Omega=X$, we will set $\capa_p(K)=\capa_p(K,X)$.

We list some properties  of the relative $p$-capacity below \cite[Theorem 2.2]{HKM06}.
 
\begin{thm}
Let $\Omega\subset X$ be open and $K\subset \Omega$ be a set. Assume $p>1$. Then the following statements are true: 

\smallskip
\textup{(i)} If $K_1\subset K_2,$ then $\capa_p(K_1, \Omega)\le \capa_p(K_2, \Omega)$.

\smallskip
\textup{(ii)} If $\Omega_1\subset \Omega_2$, then 
\[
\capa_p(K,\Omega_2)\le \capa_p(K,\Omega_1).
\] 

\smallskip
\textup{(iii)} If $K_i$ is a decreasing sequence of compact subsets of $\Omega$ with $K=\cap_{i} K_i$, then
\[
\capa_p(K,\Omega)=\lim_{i\to \infty}\capa_p(K_i,\Omega).
\]

\smallskip
\textup{(iv)} If $K=\bigcup_i K_i$, then
\[
\capa_p(K,\Omega)\le \sum_{i}\capa_p(K_i,\Omega).
\]

\smallskip
\textup{(v)} If $K_1\subset \Om_1\subset K_2\subset \Om_2\subset \cdots \subset\Om=\bigcup_i \Om_i$, then
\[
\capa_p(K_1,\Om)^{\frac{1}{1-p}}\ge \sum_{i=1}^\infty \capa_p(K_i,\Om_i)^{\frac{1}{1-p}}.
\]
\end{thm}

If  $K\subset \Omega$ is closed, we say $u\in N^{1,p}(X)$ is a {\em $p$-potential}  of $E=(K,\Omega)$ if $u$ is $p$-harmonic on $\Omega\setminus K$, $u(x)=1$ for $x\in K$ {$p$-quasi-everywhere} and $u(x)=0$ for $x\in X\setminus \Omega$ $p$-quasi-everywhere. 
\begin{rem}
(1) One can remove the $p$-quasi-everywhere as if $u,v\in N^{1,p}$ and $u=v$ q.e., then $u \sim v$.
(2) Note that a $p$-potential in $(K, \Omega)$ is H\"older continuous in $\Omega\setminus K$, but may not be continuous in $\Omega$. \end{rem}

{If $\Omega\sub X$ is a relatively compact region, $K\sub \Om$,  and $\capa_p(K,\Omega)<\infty$, then a $p$-potential for $(K,\Om)$ 
always  exists as they are solution to a Dirichlet problem and the proof in  \cite[Lemma 3.3]{HS02}, which deals with  $p$-potentials defined using $p$-energy minimizer on  $\Omega\setminus K$, can be easily adapted to Cheeger gradients.} 
The proof is the following: Let $u_i\in N^{1,p}(X)$ satisfying $0\le u\le 1$, $u=1$ q.e on $K$, $u=0$ q.e. on $X\setminus\Omega$ be a minimizing sequence such that $\int_\Omega |Du_i|^p\to \capa_p(K,\Omega)$. Since $\Omega$ is relative compact, $u_i$ is a bounded sequence in $N^{1,p}(X)$. By reflexivity and Mazur's lemma, we can get a convergent subsequence $u_i\to u$ in $L^p(X)$ and $|Du_i|\to g$ in $L^p(X)$. By passing to a subsequence, $u_i\to u$ a.e.. We can define $\tilde{u}=\limsup_{i\to \infty} u_i$, then for $p$-almost all curves, $\tilde{u}(\gamma(\ell))=\tilde{u}(\gamma(0))=\infty$ or 
\[
|\tilde{u}(\gamma(\ell))-\tilde{u}(\gamma(0))|\le \limsup_{i\to\infty} |u_i(\gamma(\ell))-u_i(\gamma(0))|\le \limsup_{i\to\infty}\int_\gamma |Du_i|=\int_\gamma g,
\] 
where the last inequality follows from Fuglede's lemma. Hence, $\tilde{u}=u$ a.e. and $\tilde{u}\in N^{1,p}(X)$ where $g$ is an $p$-weak upper gradient of $\tilde{u}$. Likewise, $\hat{u}=\liminf_{i\to \infty} u_i\in N^{1,p}(X)$ and $\hat{u}=\tilde{u}$ a.e.. This implies that $\hat{u}=\tilde{u}=u$ q.e.. We obtain that by passing to a subsequence $u_i\to \tilde{u}$ q.e. and $\int_\Omega g_{\tilde{u}}^p\le \int_\Omega g^p=\lim_{i\to \infty} |Du_i|^p\le \capa_p(K,\Omega)\le \int_\Omega g_{\tilde{u}}^p=\int_\Omega |D\tilde{u}|^p.$ It is also easy to show $\tilde{u}$ is $p$-harmonic in $\Omega\setminus K$. We conclude the proof.

To simplify the notation, we are going to use the expression $\{ u>\alpha\}$ in place of $\{ x\in \Omega \colon u(x)>\alpha\}$. We have the following lemma.
\begin{lem}\label{cap01}
Let $K$ be a compact set in $\Omega\subset X$ and let $u $ be the $p$-potential of $E=(K, \Omega)$. Suppose $u$ is continuous in $\Omega$. 
Then for all $0\le \alpha <\beta\le1$,
\[
\capa_p (\Omega\cap  \{u\ge \beta\}, \Omega\cap  
 \{u>\alpha\})=(\beta-\alpha)^{1-p} \capa_p (K,\Om).
\]
\end{lem}
\begin{proof}
To prove this lemma, we define the following function 
\[
w(x)=\begin{cases} 0 &\text{if } u(x)\le \alpha\\
\displaystyle \frac{u(x)-\alpha}{\beta-\alpha}& \text{if } \alpha<  u(x)< \beta\\
1& \text{if }   u(x)\ge \beta.\\
\end{cases}
\]
Then $w$ is 
{is continuous and} is the $p$-potential of  $E'=(\Omega\cap\{u\ge \beta\},\, \Omega\cap\{u>\alpha\})$. Since $u$ is $p$-harmonic in $\Omega\setminus K$, and $u, w$ have the same boundary values on the boundary of the ring $(K,\Om)$ we can take $\phi=u-w$ as a test function and conclude 
\[
\begin{aligned}
\capa_p (K,\Om)
&= \int_\Om |Du|^{p-2}Du\cdot D u=\int_{\Om\cap \{u>\alpha\}} |Du|^{p-2}Du\cdot D w\\
&=(\beta-\alpha)^{p-1}\int_{\Om\cap \{u>\alpha\}} |Dw|^p\\&
=(\beta-\alpha)^{p-1}\capa_p (\Omega\cap  \{u\ge \beta\}, \Omega\cap  
 \{u>\alpha\}).  \qedhere
\end{aligned}
\]
\end{proof}
 A similar result was recently proved for $p$-energy minimizers in  \cite[Theorem 3.3]{BBLe20}.

A metric measure space $(X,d,\mu)$ is {\em $p$-hyperbolic} if there is a compact set $K\sub X$ such that $\capa_p(K)>0$. On the other hand, if for each ball $B\sub X$ we have  $\capa_p(B)=0$, then we say that the space is {\em $p$-parabolic}. It is shown in \cite[Theorem 3.14]{HS02} that a space $X$ is $p$-hyperbolic if and only if for every $y\in X$ there exists a singular solution bounded from below. The characterization of $p$-parabolicity in terms of volume growth on a non-compact complete Riemannian manifold is investigated in \cite{HK01}.

Suppose $E$ and $F$ are both closed subsets of an open set $\Omega$ in $X$. The triple $(E, F; \Omega)$ is called a {\em condenser} and its $p$-capacity is defined as
\[
\capa_p(E,F; \Omega)=\inf_{u}\int {|Du|}^p\, d\mu,
\]
where the infimum is taken over all $u\in N^{1,p}(\Omega)$ such that $u(x)=1$ for $x\in E$ and $u(x)=0$ for $x\in F$. Note $\capa_p(E,F; X)=\capa_p (E, F^c)$ if $E\subset F^c$.

The following Loewner-type property \cite[Theorem 5.7, Theorem 3.6]{HK98} is essential in the proof. 

%
%
\begin{lem}\label{lem:cap lower bound}
Let $(X, d,\mu)$ be a complete doubling metric measure space supporting a Q-Poincar\'e inequality for $ Q\ge 1$. If there exists a constant $C\ge 1$ such that $C^{-1}R^Q\le \mu(B_R)$ for all $B_R\in X$ with $R\le \diam(X)$, then there exists a homeomorphism $\phi\: (0, \infty)\to (0, \infty)$ such that

\begin{equation}\label{eq:cap lower bound}
\capa_Q(E,F;X)\ge \phi\biggl(\frac{\dist(E,F)}{\min\{\diam(E), \diam(F)\}}\biggr)
\end{equation}
for all non-degenerate, disjoint continua $E$ and $F$ in $X$. 
\end{lem}

If we further assume that  $X$ is $Q$-Ahlfors regular, then more can be said about  the asymptotic behaviors of the function $\phi$ near zero and near  infinity (see \cite[Theorem 3.6]{HK98}).  

\begin{lem}\label{cap}
Let $(X, d,\mu)$ be a complete $Q$-Ahlfors regular space supporting a $Q$-Poincar\'e inequality. Then the decreasing homeomorphism $\phi\: (0, \infty)\to (0, \infty)$ in  \eqref{eq:cap lower bound} can be chosen so that $\phi(t)\asymp \log(1/t)$ as $t\to 0$ and $\phi(t)\asymp \log(t)^{1-Q}$ as $t\to \infty$. 
\end{lem}

We will use the following capacity estimates which is a special case of  Theorem 1.2 in  \cite{BBLe17}: if $(X,d. \mu)$ is an unbounded  Ahlfors $Q$-regular space supporting a $Q$-Poincar\'e inequality, then there exists a constant $C$ depending on the structure conditions such that for all $x_0\in X$ and all  $0<2r\le R$, 
\begin{equation}\label{Q metric ring}
C^{-1}\log\left(\frac{R}{r}\right)^{1-Q} \le \capa_Q(\overline{B}_r(x_0), B_R(x_0))\le C\log\left(\frac{R}{r}\right)^{1-Q}. 
\end{equation}

\section{Green functions on relatively compact domains}
In this section, we assume that $(X,d,\mu)$ is  a PI-space and fix a measurable differentiable structure on $X$. We define the Green function on a relatively compact  subregion $\Om$ of $X$ as follows.

\begin{defn}\label{bd def} 
Let $p>1$, $\Omega\sub X$ be a relatively compact region, and  $x_0\in \Omega$. 
A continuous real-valued function $u$ is called a {\em $p$-harmonic Green function on $\Om$ with singularity at $x_0\in \Omega$}  if the following conditions are satisfied:
\begin{enumerate}[label=(\roman*), font=\upshape]

\item  $u$ is (Cheeger) $p$-harmonic on $\Omega\setminus \{x_0\}$.

\smallskip
\item  $\displaystyle\lim_{x\to x_0} u(x)=\capa_p(\{x_0\}, \Omega)^{\frac{1}{1-p}}$.

\smallskip
\item  $u=0$ $p$-quasi-everywhere  on $X\setminus \Omega$.

\smallskip
\item  $\capa_p(\Omega\cap \{u\ge \beta\},\Omega\cap  \{u>\alpha\})=(\beta-\alpha)^{1-p}$ for all $0\le \alpha<\beta<\infty$.
\end{enumerate}
\end{defn}
The existence of such Green functions  has   been recently proved in \cite{BBLe20} (see Theorem 1.3 (a), Lemma 5.7, Theorem 7.2, Theorem 7.4, and Theorem 8.2 in \cite{BBLe20}).

In fact, if one can find a function satisfying conditions (i), (ii), (iii) of Definition \ref{bd def}, then we can always normalize the function to obtain condition (iv), see \cite[Theorem 9.3]{BBLe20}.

\subsection{Uniqueness}
Recall that if $\Om\sub X$ is fixed,  we use the notation $\{ u>\alpha\}$ for the set $\{ x\in \Omega \colon u(x)>\alpha\}$. We begin with an auxiliary result.
\begin{lem} \label{lem:nococo} 
Let $X$ be a PI-space, $\Omega\subset X$ be an open set,
and $u\: \Om\ra \R$ be a  continuous $p$-harmonic function with $p>1$. Then for 
each $\alpha\in \R$ the sets $\{ u>\alpha\}$ and $\{u<\alpha\}$ have no (non-empty) components that are compactly contained in $\Om$.
\end{lem} 

\begin{proof} 
To reach a contradiction, suppose $\{ u>\alpha\}$  for some $\alpha\in \R$ has a component $U$ that is compactly contained in
$\Om$. Then each point $y_0\in \partial U\ne \emptyset $ lies in $\Omega$. In each neighborhood $N$ of  $x_0$ there are points both in  
$\{ u>\alpha\}$ and in the complement of this set; in particular,  
$N$  contains points  $x,x'\in N$ for which $u(x)>\alpha$ and $u(x')\le \alpha$. Hence $u(y_0)=\alpha$, since $u$ is continuous. This shows that $u|_{\partial U}=\alpha$. 
The comparison principle  now   implies that $u\equiv \alpha$ on $U$; but we know that $u>\alpha$ on $U$. This is a contradiction showing that $\{ u>\alpha\}$ cannot have components compactly contained in $\Om$. 
The same argument also   proves the statement for $\{ u<\alpha\}$.   
\end{proof}

\begin{thm}\label{thm1}
Consider a complete doubling metric measure space $(X, d, \mu)$. Assume that 
the space $(X,d,\mu)$ is $Q$-Ahlfors regular and supports a  $Q$-Poincar\'e inequality with $Q>1$.  Let $\Omega\subset X$ be a relative compact domain,  which is regular for the Dirichlet problem for the $Q$-Laplacian. Let $x_0\in \Omega$ and $u, v$ be two $Q$-harmonic Green functions with singularity at $x_0$. If there exists $y_0\in \Omega$ such that $u(y_0)=v(y_0)$, then $u=v$ on $\Omega\setminus\{x_0\}$.
\end{thm}

We need the following lemma. 
\begin{lem}\label{bound}
Consider a complete $Q$-Ahlfors regular metric measure space $(X, d, \mu)$ with $Q>1$. 
Assume the space $(X,d,\mu)$ supports a  $Q$-Poincar\'e inequality.  Let $\Omega\subset X$ be a relative compact domain,  which is regular for the Dirichlet problem for the $Q$-Laplacian.  If $u, v$ are two $Q$-harmonic Green functions on $\Om$ with singularity at $x_0\in \Om$, then there exists a constant $C$ depending on structure conditions such that \begin{equation}\label{bounded}
|u(x)-v(x)|\le C
\end{equation}
for all $x\in \Omega\setminus\{x_0\}$. 
\end{lem}
\begin{proof} In the following, we use the notation 
 $B_r\coloneqq B(x_0, r)$, $\overline B_r  \coloneqq  \overline B(x_0, r)$,  and $\Sigma_r\coloneqq \{x\in X: d(x,x_0)=r\}$. 
 We fix $r_0>0$ such that $\dist(x_0,\partial \Om)\ge 2r_0$, 
 consider $0<r\le r_0$, and  define 
 $$m(r)=\min\{u(x):x\in \Sigma_r\}, \quad   
M(r)=\max\{u(x): x\in \Sigma_r\} .$$
The idea of the proof is to obtain precise asymptotic estimates 
for $M(r)$ and $m(r)$ in terms of $\capa_Q(\overline B_r, B_{r_0})$.
Since a similar  estimate will also be true for $v$, the  result
will follow.

We consider the set $\{u>M(r)\}\sub \Om'\coloneqq\Om\setminus\{x_0\}$. If $U$ is a component of this set, then $U$ cannot be compactly contained in $\Om'$ by Lemma~\ref{lem:nococo}. Hence $\partial U$ meets 
$\{x_0\}\cup \partial \Om$.  The hypothesis on the regularity of $\partial \Omega$ for the Dirichlet problem yields that  $u(x)\to 0$ as $x\to \partial \Om$.
Consequently $\overline U\cap \partial \Om=\emptyset$,
and so $x_0\in \partial U$. Since   
$$U\cap \Sigma_r\sub \Sigma_r \cap \{u>M(r)\}=\emptyset$$ by definition of $M(r)$, it follows that 
$U\sub B_r$. We conclude  that  $\{u>M(r)\}\sub B_r$ and that 
the set $\{u>M(r)\}$ contains $x_0$ in its boundary.

A similar argument, relying again on the regularity hypothesis for $\partial \Omega$,  shows that since  $u(x)\to \infty$  as $x\to x_0$ 
then
$\{u<m(r)\} \sub \Omega \setminus \overline B_r$,  or equivalently $\overline B_r\sub  \{u\ge m(r)\}$,   and that the boundary of 
$\{u<m(r)\}$ contains  $\partial \Om$.

In particular, we have 
\begin{equation}\label{eq:incl}
\{u>M(r)\}\sub B_r\sub \overline{B}_r\sub \{ u\ge m(r)\}.
\end{equation}

Since $u$ is a Green function, for all $0< \alpha<\beta $ we have
\begin{equation}\label{eq:levelsets}
\capa_Q(\{u\ge\beta \}, \{ u>\alpha\})=\frac{1}{(\beta-\alpha)^{Q-1}}.
\end{equation}
For small $\varepsilon>0$ we apply this for $\beta=m(r_0)-\varepsilon$ and 
$\alpha =M(r)+\varepsilon$ . Note that 
$$\{u\ge M(r)+\varepsilon\}\sub \overline B_r\sub  
B_{r_0}\sub \{u>m(r_0)-\varepsilon\}. $$ Since the second inclusion is strict here and this holds for all $\varepsilon>0$, we must have $M(r)\ge m(r_0)$. 
It follows  that 
\begin{align*}
\capa_Q( \overline B_r, B_{r_0})&\ge 
\capa_Q(\{ u\ge M(r)+\varepsilon\}, \{u> m(r_0)-\varepsilon\})\\
&= \frac{1}{(M(r)-m(r_0)+2\varepsilon)^{Q-1}}. 
\end{align*} 
 Rearranging this inequality and letting $\varepsilon\to 0$,  we arrive at 
 \begin{equation}\label{eq:uppbdd}
M(r) \ge m(r_0)+ \biggl(\frac{1}{\capa_Q(\overline B_r, B_{r_0})}\biggr)^{\frac 1{Q-1}}.
\end{equation}

Next, we show that 
\begin{equation}\label{eq:uppbdd2}
m(r) \le M(r_0)+ \biggl(\frac{1}{\capa_Q(\overline B_r, B_{r_0})}\biggr)^{\frac 1{Q-1}}.
\end{equation}
This is obviously true if $m(r)\le M(r_0)$.  
If $m(r) >M(r_0)$, then 
$$  \overline B_r\sub \{u\ge m(r)\} \sub  \{u>M(r_0)\} \sub B_{r_0}$$ 
and it follows that 
\begin{align*}
\capa_Q( \overline B_r, B_{r_0})&\le 
\capa_Q(\{ u\ge m(r)\}, \{u> M(r_0)\})\\
&= \frac{1}{(m(r)-M(r_0))^{Q-1}},
\end{align*} 
thus yielding
\eqref{eq:uppbdd2}.
To obtain effective control for the growth of the function $u$ we need a bound for $M(r)-m(r)$. 

\smallskip
\noindent {\em Claim.} There exists a constant $C>0$ such that 
$M(r)-m(r)\le C$ for all $r>0$. 

\smallskip
 To prove the claim, observe that there exist point $y_1,y_2\in \Sigma_r$ such that $u(y_1)=M(r)$ and $u(y_2)=m(r)$.
If $M(r)=m(r)$, there is nothing to prove. Otherwise, $M(r)>m(r)$  and
 we consider a small $\varepsilon>0$ such that 
 $M(r)-\varepsilon> m(r)+\varepsilon$.
 Then $y_1\in \{u>M(r)-\eps\}$ and $y_2\in \{u<m(r)+\varepsilon\}$.
 Let $U$ be the connected component of   $\{u>M(r)-\varepsilon\}$  
 that contains $y_1$ and $V$ be the connected component 
 of $\{u<m(r)+\varepsilon\}$ that contains $y_2$. By what we have seen, 
 $x_0\in \partial U$ and $\partial V\cap \partial\Om \ne \emptyset$.
This implies that $\diam(U)\ge d(x_0, y_1)=r$ and 
\begin{align*}\diam(V)&\ge \dist (y_2, \partial \Om)\\
&\ge \dist (x_0, \partial \Om)-d(x_0, y_2)\ge 2r_0-r\ge r_0. 
 \end{align*}
 Define $E=\{u\ge M(r)-\varepsilon\}$ and $F=\{ u\le m(r)+\varepsilon\}$.
 Then $y_1\in U\sub E$ and $y_2\in V\sub F$. This implies that 
 $$ \diam(E)\wedge \diam (F) \ge   \diam(U)\wedge \diam (V)\ge r. $$
 On the other hand, 
 $\dist(E,F)\le  d(y_1, y_2)\le 2r.$
 By Lemma \ref{lem:cap lower bound}, we get that 
\[
\capa_Q(E,F; \Omega)\ge C
\]
with a constant $C>0$ independent of $E$ and $F$, and hence 
independent of $r$ and $\varepsilon$. On the other hand, 
\begin{align*}C\le \capa_Q(E,F; \Omega)&=\capa_Q(\{u\ge M(r)-\varepsilon\}, \{u>m(r)+\varepsilon\})\\
&=\frac{1}{(M(r)-m(r)-2\varepsilon)^{Q-1}}. 
\end{align*} 
Rearranging the latter and letting $\varepsilon\to 0$,  yields the desired claim (with a suitable constant independent of $r$ and $u$).

Now let  $x\in \Om$ be arbitrary. If $0<d(x,x_0)< r_0$, then we can apply all the previous estimates
for $r\coloneqq d(x_0,x)<r_0$. 
Then we see that 
\begin{align*} 
u(x) & \ge m(r) \ge M(r) -C\\
 &\ge {\capa_Q(\overline B_r, B_{r_0})}^{\frac 1{1-Q}}+m(r_0)-C. 
\end{align*}
and  
\begin{align*} 
u(x) & \le M(r) \le m(r) +C\\
 &\le {\capa_Q(\overline B_r, B_{r_0})}^{\frac 1{1-Q}}+M(r_0)+C.
\end{align*}
In summary, we obtain 
\begin{equation}\label{capmn1}
{\capa_Q(\overline B_r, B_{r_0})}^{\frac 1{1-Q}}+M(r_0)-2C\le u(x)\le {\capa_Q(\overline B_r, B_{r_0})}^{\frac 1{1-Q}}+m(r_0)+2C
\end{equation}

Moreover, $\{u\ge M(r_0)\}\subset \overline{B}_{r_0}\subset \{ u\ge m(r_0)\}$ implies that
\[
\capa_Q (\{u\ge M(r_0)\}, \Omega)\le \capa_Q(\overline{B}_{r_0} ,\Omega)\le  \capa_Q(\{u\ge m(r_0)\}, \Omega),
\]
while 
\[
\capa_Q (\{u\ge M(r_0)\}, \Omega)=M(r_0)^{1-Q} \quad{and}\quad \capa_Q (\{u\ge m(r_0)\}, \Omega)=m(r_0)^{1-Q} .
\]
Hence, we obtain that $$m(r_0)\le \capa_Q (\overline{B}_{r_0}, \Omega)^{\frac{1}{1-Q}}\le M(r_0).$$ Combining with \eqref{capmn1}, we get
\begin{align*}
&{\capa_Q(\overline B_r, B_{r_0})}^{\frac 1{1-Q}}+\capa_Q (\overline{B}_{r_0}, \Omega)^{\frac{1}{1-Q}}-2C\\
\le &u(x)\le {\capa_Q(\overline B_r, B_{r_0})}^{\frac 1{1-Q}}+\capa_Q (\overline{B}_{r_0}, \Omega)^{\frac{1}{1-Q}}+2C.
\end{align*}
In other words, $u(x)$ agrees with  ${\capa_Q(\overline B_r, B_{r_0})}^{\frac 1{1-Q}} $ up to a uniformly bounded additive term with $r=d(x,x_0)<r_0$. If $d(x,x_0)\ge r_0$, we argue similarly and get that $u(x)$ agrees with  $-{\capa_Q(\overline B_{r_0}, B_{R})}^{\frac 1{1-Q}} $ up to a uniformly bounded additive term with $R=d(x,x_0)$.

If $v$ is another $Q$-harmonic Green function in $\Om$ with pole $x_0$, then all the previous considerations apply and we obtain similar estimates. It follows that there exists a constant $C>0$ depending on the structure conditions such that 
$|u(x)-v(x)|\le C$ on $\Om\setminus \{x_0\}$. 
\end{proof} 

The proof implies in particular the following 

\begin{cor}\label{bound-bis} In the hypotheses of Lemma \ref{bound}, if $u$ is a  $Q$-harmonic Green function on $\Om$ with singularity at $x_0\in \Om$, then there exists a constant $C>0$ depending only on the structure constants, such that  for $0<r<R$ one has \begin{equation}\label{bounded-bis1}
\bigg|u(x)-u(y) - {\capa_Q( \overline B_r, B_{R})}^{\frac 1{1-Q}}\bigg|\le C
\end{equation}
for all $x,y\in \Omega\setminus\{x_0\}$, such that $r=d(x_0,x)$ and $R=d(x_0,y)$. 

\end{cor}

Next,  we can complete the proof of Theorem \ref{thm1}, the uniqueness of $Q$-harmonic Green function on $\Om$ with singularity at $x_0\in \Omega$ . 
\begin{proof}
We argue by contradiction. Assume there exists $z_0\in \Omega$ such that $u(z_0)>v(z_0)$ and a constant $\lambda\in (0,1)$ such that $\lambda u(z_0)> v(z_0)$. Consider the set $E=\{x\in \Omega: \lambda u(x)> v(x)\}$. We claim that $x_0\in \partial E_{z_0}$, where $E_{z_0}$ is the connected component of $z_0$ in $E$. If not, then $u$ is $Q$-harmonic on $E_{z_0}$ and $\lambda u=v$ on $\partial E_{z_0}$. Comparison principle implies that $\lambda u(x)=v(x)$ for all $x\in E_{z_0}$. This contradicts the fact that $\lambda u(z_0)> v(z_0)$ and $z_0\in E_{z_0}$. Then we have
\[
\liminf_{x\to x_0} \frac{v(x)}{u(x)}\le \lambda<1.
\]
On the other hand, Lemma \ref{bound} implies that 
\[
\lim_{x\to x_0}\frac{v(x)}{u(x)}=1,
\]
yielding a contradiction.
\end{proof}


\section{Globally defined Green functions}
We assume that $(X,d,\mu)$ is  an unbounded  $Q$-Ahlfors regular PI-space and fix a measurable differentiable structure. We will prove existence of global Green functions and obtain the uniqueness for $Q>1$ based on the following definition.

\begin{defn}\label{global-Green}
Let $x_0\in X$ and $Q>1$. A continuous real-valued function $u\: X\ra \R$ is called a {\em global $Q$-harmonic Green function with singularity at $x_0\in X$} if the following conditions are satisfied:
\begin{enumerate}[label=(\roman*), font=\upshape]

\smallskip
\item $u$ is $Q$-harmonic in $X\setminus \{x_0\}$. 

\smallskip
\item $\displaystyle\lim_{d(x, x_0)\to 0} u(x)=+\infty$.

\smallskip
\item $\displaystyle\lim_{d(x,x_0)\to \infty} u(x)=-\infty$.

\smallskip
\item $\C(\{u\ge \beta\}, \{u>\alpha\})=(\beta-\alpha)^{1-Q}$, for all $-\infty<\alpha<\beta<\infty$.

\end{enumerate}
\end{defn}

\subsection{Existence}
\begin{thm}\label{global existence}
Let $(X, d, \mu)$ be an unbounded, complete $Q$-Ahlfors regular metric measure space supporting a $Q$-Poincar\'e inequality for $Q>1$. Let $x_0\in X$. Then there exists a global $Q$-harmonic Green function with singularity at $x_0$ in $X$.
\end{thm}

We first show that if we can find a function satisfying conditions (i), (ii), (iii) of Definition \ref{global-Green}, then we can always normalize the function to obtain condition (iv). The proof is similar to that of Lemma \ref{cap01}.

\begin{prop}[Normalization]\label{normalization}
Under the assumptions Theorem \ref{global existence}, suppose that there exists a continuous function  $g\:X\to \mathbb{R}$ satisfying  \textnormal{(i)}, \textnormal{(ii)},
\textnormal{(iii)} in Definition~\ref{global-Green}. Then there exists $\lambda>0$ such that $\lambda g$ is a global $Q$-harmonic Green function. 
\end{prop}
\begin{proof}

It suffices to show that  $(\beta-\alpha)^{Q-1} \Capab $
is a constant $\lambda^{1-Q}$ independent of the interval $(\alpha,\beta)$, and  depending only on $g$ and on the structure constants,  i.e., 
$$ \Capab=
\frac{\lambda^{1-Q}} {(\beta-\alpha)^{Q-1}} $$
for all $\alpha<\beta$.
Note that $u=\min \{(g-\alpha)_+/(\beta-\alpha), 1\}$ is a $Q$-potential for the condenser
$\{u\ge \beta\}, \{u>\alpha\})$, and so
$$  (\beta-\alpha)^{Q-1} \Capab=
\frac{1} {\beta-\alpha}\int_{\{\alpha <g < \beta\}}|D g|^Q.
$$

Consider $\alpha <\alpha'<\beta'< \beta$, and let 
$v=\max\{ (g-\alpha')_+/(\beta'-\alpha'), 1\}$. Observing that $u-v=0$ outside $\{\alpha <g < \beta\}$, and that 
$g$ is $Q$-harmonic on $\{\alpha <g < \beta\}$ one has
\begin{align*}0&=\int_{\{\alpha <g < \beta\}} |D g|^{Q-2} D g \cdot D (u-v)\\
&=\frac 1{\beta-\alpha}\int_{\{\alpha <g < \beta\}}|D g|^Q-\frac 1{\beta'-\alpha'}\int_{\{\alpha' <g < \beta'\}}|D g|^Q\\
&=(\beta-\alpha)^{Q-1}\Capab-(\beta'-\alpha')^{Q-1}\Capabp.
\end{align*}
The statement follows. 
\end{proof}

The strategy of the proof for existence of a global Green function is as follows: With the existence of $p$-potential functions guaranteed by \cite[Lemma 3.3]{HS02}, we construct a sequence of $Q$-potential functions $\{u_i\}$ on the sequence of nested rings  $B(x_0, 2^i) \setminus B(x_0, 2^{-i})$  that exhaust $X\setminus \{x_0\}$ as $i\to \infty$.  Next, we wish to invoke the Arzela-Ascoli theorem to obtain a sequence of $Q$-harmonic functions that converges locally uniformly on annuli. 

\begin{proof}[Proof of Theorem \ref{global existence}]
Let $D_i=B(x_0, 2^i)$ and $B_i=B(x_0, 2^{-i})$ for $i\ge 2$. Let $\tilde{u}_i$ be the $Q$-potential of $(\overline{B_i}, D_i)$.

We next construct a sequence of functions $\{v_i\}$ such that 

\begin{itemize}
\smallskip
\item[(i)] $v_i$ is a $Q$-harmonic function on $D_i\setminus \overline{B_i}$,

\smallskip
\item[(ii)] For any fixed $i_0\ge 2$, $v_i$ is locally uniformly bounded on $\overline{D_{i_0}}\setminus B_{i_0}$ for sufficiently large $i> i_0$.
\end{itemize}

By Lemma \ref{cap01}, for all $0<\alpha<\beta< 1$,

\[
\capa_Q (\Omega\cap  \{u\ge \beta\}, \Omega\cap  
 \{u>\alpha\})=(\beta-\alpha)^{1-Q} \capa_Q (K,\Omega),
\]
where $u$ is a $Q$-potential of $(K,\Omega)$. Hence, similar to the proof of Proposition \ref{normalization}, we can normalize each $\tilde{u}_i$ by multiplying a constant $\lambda_i=\capa_Q (\overline{B}_i,D_i)^{\frac{1}{1-Q}}$ such that
\begin{equation}\label{norm}
     \capa_Q (D_i\cap  \{u_i\ge \beta\}, D_i\cap  
 \{u_i>\alpha\})=(\beta-\alpha)^{1-Q},
\end{equation}
where $u_i=\lambda_i \tilde{u}_i$.

 Next, apply the argument similar in Lemma \ref{bound} to each $u_i$ on the open set $D_i\setminus \overline{B_i}$. 
For each $i_0\in \mathbb{N}$ and $i>i_0$ sufficiently large. Consider $2^{-i_0}\le r<R\le 2^{i_0}$ and 
$$m(r)=\min\{u_i(x):x\in \Sigma_r\}, \quad   
M(r)=\max\{u_i(x): x\in \Sigma_r\} ,$$
where $\Sigma_r\coloneqq \{x\in X: d(x,x_0)=r\}$. Lemma \ref{lem:nococo}  applies to $u_i$ and for $\alpha\in (0,1)$, we have $\{u_i>\alpha\}$ and $\{u_i<\alpha\}$ have no non-empty components that are compactly contained in $D_i\setminus \overline{B_i}$. In particular, any connected component of $\{u_i>M(r)\}\subset B_r$ contains points in $\partial B_i$ in its boundary, any connected component of $\{u_i<m(r)\}$ contains points in $\partial D_i$ in its boundary, and
\[
\{u_i>M(r)\}\subset B_r\subset \overline{B}_r\subset \{u_i\ge m(r)\}.
\]

Following the same argument as in Lemma \ref{bound}, we obtain the same estimates \eqref{eq:uppbdd} and \eqref{eq:uppbdd2} for $r<R$ that
\[
M(r)\ge m(R)+\capa_Q(\overline{B}_r, B_R)^{\frac{1}{1-Q}}
\]
and
\[
m(r)\le M(R)+\capa_Q(\overline{B}_r, B_R)^{\frac{1}{1-Q}}.
\]
We next show that there exists a constant $C>0$ such that $M(r)-m(r)\le C$ for $u_i$ with $i$ sufficiently large on $D_{i_0}\setminus \overline{B}_{i_0}$ and $C$ depending on the structure constants of $X$ only.

Assume $M(r)>m(r)$ and let $\varepsilon>0$ be such that $M(r)-\varepsilon>m(r)+\varepsilon$. There exists $y_1,y_2\in \Sigma_r$ such that $u(y_1)=M(r)$ and $u(y_2)=m(r)$. Let $U$ be a connected component of $\{u>M(r)-\varepsilon\}$ containing $y_1$ and $V$ be a connected component of $\{u<m(r)+\varepsilon\}$ containing $y_2$. Hence,

\[
{\rm diam}(U)\ge r-2^{-i}\ge \frac{r}{2}\quad\quad \text{and}\quad\quad {\rm diam}(V)\ge 2^{i}-r\ge 2r,
\]
for $i>i_0$ sufficiently large and $r\in [2^{-i_0}, 2^{i_0}]$.

Let $E=\{u>M(r)-\varepsilon\}$ and $F=\{u<m(r)+\varepsilon\}$. Then ${\rm dist}(E,F)\le d(y_1,y_2)\le 2r$ and
\[
{\rm diam}(E)\wedge {\rm diam}(F)\ge {\rm diam}(U)\wedge {\rm diam}(V)\ge \frac{r}{2}.
\]
Lemma \ref{lem:cap lower bound} implies that $\capa_Q(E,F;D_i)\ge C$, where $C$ is determined by the structure condition of the PI space $X$ and independent of $u_i$. 

Furthermore, the normalization \eqref{norm} implies that
\[
\capa_Q(E,F;D_i)= \capa_Q(\{u>M(r)-\varepsilon\},\{u<m(r)+\varepsilon\};D_i)=(M(r)-m(r)-2\varepsilon)^{1-Q}.
\]
Letting $\varepsilon\to 0$ yields the claim that $M(r)-m(r)\le C$ for $u_i$ with $i$ sufficiently large for $r$ in $D_{i_0}\setminus B_{i_0}$ and $C$ depending on the structure constants of $X$ only.

For each $u_i$, we take a point $z_i\in \Sigma_1$ and define $v_i=u_i-u_i(z_i)$. Then $v_i(z_i)=0$ and we claim that $v_i$ is uniformly bounded on $D_{i_0}\setminus \overline{B}_{i_0}$ for each $i_0$ and $i>i_0$ sufficiently large.

For $y\in D_{i_0}\setminus \overline{B}_{i_0}$ and $d(x_0,y)>1$. Let $R=d(x_0,y)$ and $y\in \Sigma_R$ with $R\in (1, 2^{i_0}]$. Let $r=1$, $r<R\le 2^{i_0}$ and $x\in \Sigma_r, y\in \Sigma_R.$ The oscillation bound of $u_i$ on $\Sigma_r$ and $\Sigma_R$ together with  \eqref{eq:uppbdd} and  \eqref{eq:uppbdd2} imply that
\[
\begin{aligned}
    u_i(y)\le M(R)&\le m(R)+C\le M(r)-\capa_Q(\overline{B}_r, B_R)^{\frac{1}{1-Q}}+C\\
    &\le m(r)-\capa_Q(\overline{B}_r, B_R)^{\frac{1}{1-Q}}+2C\\
    &\le u_i(x)-\capa_Q(\overline{B}_r, B_R)^{\frac{1}{1-Q}}+2C,
\end{aligned}
\]
and similarily,
\[
\begin{aligned}
    u_i(y)\ge  u_i(x)-\capa_Q(\overline{B}_r, B_R)^{\frac{1}{1-Q}}-2C.
\end{aligned}
\]
Let $x=z_i$ and it follows that 
\[
-2C-\capa_Q(\overline{B}_r, B_R)^{\frac{1}{1-Q}}\le v_i(y) \le 2C-\capa_Q(\overline{B}_r, B_R)^{\frac{1}{1-Q}}
\]
for any $y\in \Sigma_R$ with $R\in (1, 2^{i_0}]$. Recall $B_{i_0}=B(x_0, 2^{-i_0})$, $r=1>2^{-i_0}$ and $D_{i_0}=2^{i_0}$. From the above estimates, it is easy to see that for any $y\in D_{i_0}\setminus \overline{B}_{1}$,
\[
-2C-\capa_Q(\overline{B}_{i_0}, D_{i_0})^{\frac{1}{1-Q}}\le v_i(y) \le 2C.
\]
Likewise, when $y\in D_{i_0}\setminus \overline{B}_{i_0}$ satisfying $d(x_0,y)<1$, we let $R=1$ and $r=d(x_0,y)$ and apply similar argument to get a uniform bound of $v_i(y)$ for $y\in \Sigma_r$ that
\[
-2C\le v_i(y)\le 2C+\capa_Q(\overline{B}_{i_0}, D_{i_0})^{\frac{1}{1-Q}}.
\]

In summary, we obtain a sequence $v_i$ which is equicontinuous and uniformly bounded on $D_{i_0}\setminus \overline{B}_{i_0}$ for each $i_0$ for $i>i_0$ sufficiently large. Let $i_0\to \infty$, Arzela-Ascoli Theorem and a diagonal argument extracts a subsequence such that converges locally uniformly to a $Q$-harmonic function $u$ on $X\setminus\{x_0\}$. 

The space $X$ is $Q$-parabolic follows from a direct computation. Since $X$ is a $Q$-Ahlfors regular space, the following capacity estimate of rings yields
\[
\capa_Q(B_r, B_{2r})\le C\mu(B_r) r^{-Q}\le C.
\]

For $i> i_0$ sufficiently large, we apply \cite[Theorem 2.6]{HKM06} and get
\[
\begin{aligned}
{\capa_Q(\overline{D}_{i_0}, D_i)}^{\frac{1}{1-Q}}&\ge \sum_{j=i_0}^{i-1}{\capa_Q (\overline{B}(x_0, 2^j), B(x_0, 2^{j+1}))}^{\frac{1}{1-Q}}\ge \sum_{j=i_0}^{i-1} C=C(i-i_0).
\end{aligned}
\]
It follows that 
\[
{\capa_Q(\overline{D}_{i_0}, D_i)}\le \frac{C}{(i-i_0)^{Q-1}}.
\]
Let $i\to \infty$, we obtain the space is $Q$-parabolic. 

We then verify that the $Q$-harmonic function on $X\setminus \{x_0\}$ obtained above satisfies condition (3) and condition $(2)$. 
According to the characterization in \cite[Theorem 9.22]{HKM06}, the $Q$-harmonic function obtained above is unbounded from below. That is, $\inf u=-\infty$. Since $u$ is the locally uniform limit of the sequence $\{v_i\}$ and $\{v_i\}$ are bounded from below on $D_{i_0}\setminus \overline{B}_{i_0}$, the infimum of $u$ can only be achieved as $d(x, x_0)\to \infty$. Next, we show that 
\[
\lim_{d(x,x_0)\to \infty} u(x)=-\infty.
\]

For $i$ sufficiently large, the locally uniform convergence of $v_i$ to $u$ and locally uniform boundedness of $v_i$ imply that the oscillation of $u$ on is also locally uniformly bounded. There exists a sequence $x_j\to \infty$ such that $u(x_j)\to -\infty$. It follows that $u(x)\to -\infty$ for $d(x,x_0) \to\infty $. That is, $\displaystyle\lim_{d(x,x_0)\to \infty} u(x)=-\infty.$

On the other hand, 
if $\sup u<\infty$ in $B(x_0, r)\setminus\{x_0\}$, then $u$ can be extended to a non-constant $Q$-harmonic function on $B(x_0, r)$ and achieves a maximum at an interior point, which is contradiction to the maximum principle. Thus, $\sup u=\infty$ on $X\setminus \{x_0\}$. Furthermore, $u$ is the locally uniform limit of the sequence $\{v_i\}$ and $\{v_i\}$ are uniformly bounded from above on $D_{i_0}\setminus \overline{B}_{i_0}$ so we can only achieve $\sup u$ as $x\to x_0$. A similar argument as in the case of $x\to \infty$ yields that
\[
\lim_{d(x, x_0)\to 0} u(x)=\infty.
\]

To conclude the proof of existence for a global Green function  we invoke Proposition \ref{normalization}.
\end{proof}

\subsection{Uniqueness}

We show that a global $Q$-harmonic Green function is unique in $X$ for $1<Q<\infty$.

\begin{thm}\label{thm2}
Let $(X, d, \mu)$ be an unbounded, complete $Q$-Ahlfors regular measure space supporting a $Q$-Poincar\'e inequality for $Q >1$. Let $x_0\in X$.
 Let $u, v$ be two global $Q$-harmonic Green functions with singularity at $x_0$ in $X$. If $u,v$ agree at a point $y_0\in X{\setminus\{x_0\}}$, then  $u=v$ on $X\setminus\{x_0\}$.
\end{thm}

In the following, we use the notation 
 $B_r\coloneqq B(x_0, r)$, $\overline B_r  \coloneqq  \overline B(x_0, r)$,  and $\Sigma_r\coloneqq \{x\in X: d(x,x_0)=r\}$. We also define  $$m(r)=\min\{u(x):x\in \Sigma_r\}, \quad   
M(r)=\max\{u(x): x\in \Sigma_r\} .$$

A simple modification of the proof of Lemma \ref{bound} yields the following result. 
\begin{lem}\label{global bound}
Let $(X, d, \mu)$ be a complete $Q$-Ahlfors regular measure space supporting a $Q$-Poincar\'e inequality for $Q>1$. Let $x_0\in X$ and let $u,v$ be two global $Q$-harmonic Green functions with singularity at $x_0$.  Then there exists a constant $C_0$ depending on the structure conditions and maximum of $u-v$ on a fixed sphere such that \begin{equation}\label{eq:bounded}
|u(x)-v(x)|\le C_0
\end{equation}
for all $x$ satisfying $0<d(x,x_0)<\infty$. 
\end{lem}

\begin{proof} 
 Fix $r_0>0$. Let $0<r\le r_0$. We consider the set $\{u>M(r)\}\subset X\setminus\{x_0\}$. If $U$ is a component of this set, then $U$ cannot be compactly contained in $X\setminus\{x_0\}$ by Lemma~\ref{lem:nococo}. Hence $\partial U$ meets 
$\{x_0\}\cup \{\infty\}$.  Since  $u(x)\to -\infty$ as $x\to \infty$.
Consequently $x_0\in \partial U$. Since   
$$U\cap \Sigma_r\subset \Sigma_r \cap \{u>M(r)\}=\emptyset$$ by definition of $M(r)$, it follows that 
$U\subset B_r$. We conclude  that  $\{u>M(r)\}\subset B_r$ and that 
the set $\{u>M(r)\}$ contains $x_0$ in its boundary. 

Similarly, using the fact that $u(x)\to \infty$  as $x\to x_0$, yields 
that
$\{u<m(r)\} \subset X \setminus \overline B_r$,  or equivalently $\overline B_r\sub  \{u\ge m(r)\}$,   and that the boundary of 
$\{u<m(r)\}$ contains  $\{\infty\}$.

Following the same argument in Lemma 3.4, we obtain that
 \begin{equation}\label{global eq:lowbdd}
M(r) \ge m(r_0)+ {\capa_Q(\overline B_r, B_{R})}^{\frac 1{1-Q}}.
\end{equation}

and  
 \begin{equation}\label{global eq:uppbdd}
m(r) \le M(r_0)+ {\capa_Q(\overline B_r, B_{R})}^{\frac 1{1-Q}}.
\end{equation}

Similar argument as in Lemma \ref{bound} implies there exists a constant $C>0$ depending only on the structure conditions of $X$, such that 
$M(r)-m(r)\le C$ for all $0<r<\infty$. 

If $0<d(x,x_0)\le r_0$, then let $r= d(x_0,x)$ and we see that 
\begin{equation}\label{function value lower bound I}
\begin{aligned} 
u(x) & \ge m(r) \ge M(r) -C\\
 &\ge {\capa_Q(\overline B_r, B_{r_0})}^{\frac 1{1-Q}}+m(r_0)-C. 
\end{aligned}
\end{equation}
and  
\begin{equation}\label{function value upper bound I}
\begin{aligned} 
u(x) & \le M(r) \le m(r) +C\\
 &\le {\capa_Q(\overline B_r, B_{r_0})}^{\frac 1{1-Q}}+M(r_0)+C. 
\end{aligned}
\end{equation}
In other words, $u(x)$ agrees with  ${\capa_Q(\overline B_r, B_{r_0})}^{\frac{1}{1-Q}} $ 
up to a uniformly bounded additive term, that is, 
\[
{\capa_Q(\overline B_r, B_{r_0})}^{\frac 1{1-Q}}+m(r_0)-2C\le u(x)\le {\capa_Q(\overline B_r, B_{r_0})}^{\frac 1{1-Q}}+m(r_0)+2C.
\]
Likewise, one can obtain similar conclusion if $d(x,x_0)>r_0$. 

If $v$ is another Green function in $\Om$ with pole $x_0$, then all the previous considerations apply and we obtain similar estimates, 
\[
{\capa_Q(\overline B_r, B_{r_0})}^{\frac 1{1-Q}}+m'(r_0)-2C\le v(x)\le {\capa_Q(\overline B_r, B_{r_0})}^{\frac 1{1-Q}}+m'(r_0)+2C,
\]
where $m'(r_0)=\min\{v(x):x\in \Sigma_r\}.$ Let $c_0=\max\{u(x)-v(y): x,y\in \Sigma_{r_0}\}$ and $C_0=c_0+4C$. 


 In summary, we conclude that $|u(x)-v(x)|\le C_0$ for all $x\in X\setminus\{x_0\}$ satisfying $d(x,x_0)<\infty$.
\end{proof} 

As in the relatively compact case, the proof of the previous theorem implies immediately the following corollary

\begin{cor}\label{bound-bis-global} In the hypotheses of Lemma \ref{global bound}, if $u$ is a  $Q$-harmonic Green function in $X$ with singularity at $x_0\in X$, then there exists a constant $C>0$ depending only on the structure constants, such that  for $0<r<R$ one has \begin{equation}\label{bounded-bis2}
\bigg|u(x)-u(y) - {\capa_Q( \overline B_r, B_{R})}^{\frac 1{1-Q}}\bigg|\le C
\end{equation}
for all $x,y\in \Omega\setminus\{x_0\}$, such that $r=d(x_0,x)$ and $R=d(x_0,y)$. 
\end{cor}
Combining with \eqref{Q metric ring}, Corollary \ref{bound-bis-global} implies the following result, which completes the proof of Theorem \ref{thm:ubddGreen} part (ii).

\begin{cor}\label{cor-growth}
Let $(X, d, \mu)$ be a unbounded, complete $Q$-Ahlfors regular metric measure space supporting a $Q$-Poincar\'e inequality for $Q>1$. Let $x_0\in X$. Let $u$ be a global $Q$-harmonic Green function with singularity at $x_0$. Then there exists a constant $C\ge 1$ depending on the structure constants of $X$ and the value of $u$ on a fixed radius such that for all $x\in X\setminus \{x_0\}$, we have
 \[
 C^{-1}\log\bigg(\frac{1}{d(x,x_0)}\bigg)-C\le u(x)\le C\log\bigg(\frac{1}{d(x,x_0)}\bigg)+C.
 \]

\end{cor}

Next, we show that the limits $\displaystyle\lim_{d(x,x_0)\to 0}(u(x)-v(x))$ and $\displaystyle\lim_{d(x,x_0)\to \infty}(u(x)-v(x))$ exist. Rather than relying on Caccioppoli-type estimates, as in the rest of the literature, we use Clarkson inequalities (see for instance \cite[Theorem 2.38]{Adams}) to bound 
the $L^Q$-norm of  $Du-Dv$. 
\begin{lem}[Clarkson's inequalities]\label{clarkson-lemma}
Let $1<p< \infty$, and $p'=p/(p-1)$. 
For $1<p\le 2$ one has 
 $$  \Vert\tfrac12 (f-g)\Vert_p^{p'}+ \Vert\tfrac12 (f+g)\Vert_p^{p'}\le 
  \Big(\tfrac12 (\Vert f\Vert_p^p+\Vert g\Vert_p^p)\Big)^{p'/p},$$
  and for $2\le p < \infty$ one has
  $$\Vert\tfrac12 (f-g)\Vert_p^p+ \Vert\tfrac12 (f+g)\Vert_p^p\le 
  \tfrac12 (\Vert f\Vert_p^p+\Vert g\Vert_p^p),$$
  for all $\R^n$-valued  $L^p$-functions $f$ and $g$.
\end{lem}

As we mentioned in the introduction, while in the proof the PDE \eqref{cpl1} is not explicitly used, the  Cheeger structure is implicitly needed in the proof of the Clarkson inequality, as one needs an inner product space structure on the space of gradients.

In the following we will use the letter $C$ to denote constants depending only on the structure conditions, that possibly may change from line to line.

\begin{lem}\label{growth est1} Let $1<p< \infty$, and $p'=p/(p-1)$. Suppose $u, v$ are two global Green functions with pole at $x_0\in X$. Then $u,v$ are normalized such that for each $-\infty<\alpha<\beta<\infty$ one has $$ \int_{\{\alpha< u<\beta\}} |Du|^p=M:=\beta-\alpha=
\int_{\{\alpha< v<\beta\}} |D v|^p.$$ There exists a constant $C$ depending on $p$ and $C_0$ from Lemma \ref{global bound} such that
$$   \int_{\{\alpha<u<\beta\}} |D(u-v)/2|^p \le C M^\delta+C$$
with $\delta=1-p/p'\in [0,1)$ when $p\in(1,2)$ and $\delta=0$ when $p\in [2,\infty)$.
\end{lem}

\begin{proof}
By virtue of Lemma \ref{global bound}, we know that for some constant $C_0>0$
$$ |u(x)-v(x)|\le C_0.$$
Consequently, one has 
$$ \int_{\{\alpha<  u<\beta\}} |D v|^p\le  \int_{\{\alpha-C<  v<\beta+C\}} |D v|^p = \beta-\alpha+2C=M+2C_0.
$$

 We consider the following two cases.
 
 \smallskip
 {\em Case 1:}  $M\le 2C_0p$.
 
 If $1<p\le 2$, the first Clarkson's inequality yields 
 
 \begin{align*}
\biggl( \int_{\{\alpha<  u<\beta\}}|D(u-v)/2|^p\biggr)^{p'/p}\le &\biggl( \int_{\{\alpha<  u<\beta\}}|D(u-v)/2|^p\biggr)^{p'/p}+\biggl( \int_{\{\alpha<  u<\beta\}} |D (u+v)/2|^p\biggr)^{p'/p}\\
 \le &\biggl(  \frac12   \biggl( \int_{\{\alpha<  u<\beta\}} |Du|^p +   \int_{\{\alpha<  u<\beta\}} |Dv|^p \biggr)    \biggr)^{p'/p}\\
 \le & (M+C_0)^{p'/p}.
  \end{align*}

This implies immediately
\[
\int_{\{\alpha<  u<\beta\}}|D(u-v)/2|^p\le 3pC_0\le CM^\delta+C.
\]

  \smallskip
  
  If $2\le p < \infty$, we use the second of the Clarkson inequalities, obtaining 
  
\begin{align*}
 \int_{\{\alpha<  u<\beta\} }|D (u-v)/2|^p&\le (M+C_0) \le 3pC_0  \le C M^\delta+C.
  \end{align*}
  \smallskip
  
 {\em Case 2:}  $M>2pC_0$.
Note that
$$\frac{u+v}{2}\ge \beta-C_0 \text { on } \{ u\ge \beta\}\quad\quad{and}\quad\quad\frac{u+v}{2}\le \alpha+C_0 \text { on } \{ u\le \alpha\}.$$
Let $g=(u+v)/2$ and $\tilde{g}=\frac{(g-\alpha-C_0)_+}{\beta-\alpha-2C_0}$.  
It follows that $\tilde{g}\ge 1$ on  $\{ u\ge \beta\}$ and $\tilde{g}=0$ on 
$\{ u\le \alpha\}$. Hence 
\begin{align*}
\int_{\{\alpha<  u<\beta\}} |D (u+v)/2|^p
&\ge (\beta-\alpha-2C_0)^p{{\rm cap}_p }(\{u\ge \beta\}, \{u>\alpha\})\\
&= (M-2C_0)^p M^{1-p}\\
&\ge  M(1-2C_0/M)^p\ge M-2pC_0.
 \end{align*}

The last inequality follows from an elementary observation. Consider a function $f(t)=(1-2C_0 t)^p$ for $t\in [0,\frac{1}{2C_0}]$. Then $f$ is convex and $f(t)\ge 1-2pC_0t$ since the right hand side is the tangent line of $f$ at $0$. Plugging in $t=\frac{1}{M}$, we obtain the last inequality. 

 In the range $1<p\le 2$, the first Clarkson's inequality yields 
 \begin{align*}
\biggl( \int_{\{\alpha<  u<\beta\}}|D(u-v)/2|^p\biggr)^{p'/p}&+\biggl( \int_{\{\alpha<  u<\beta\}} |D (u+v)/2|^p\biggr)^{p'/p}\\
& \le \biggl(  \frac12   \biggl( \int_{\{\alpha<  u<\beta\}} |Du|^p +   \int_{\{\alpha<  u<\beta\}} |Dv|^p \biggr)    \biggr)^{p'/p}.
  \end{align*}

The latter implies

\begin{align*}
 \biggl(  \int_{\{\alpha<  u<\beta\} }|D (u-v)/2|^p\biggr)^{p'/p}&\le (M+2pC_0)^{p'/p}-(M-2pC_0)^{p'/p}\\
 &\le M^{p'/p} ((1+2pC_0/M)^{p'/p}-(1-2pC_0/M)^{p'/p}).
  \end{align*}
  
Let $r=\frac{2pC_0}{M}<1$ and $\alpha=p'/{p}>1$. Applying general binomial formula, we obtain
\[
\Big((1+r)^\alpha-(1-r)^\alpha\Big)^{\frac{1}{\alpha}}\le C\Big(r+r^3+\cdots r^k\Big)^{\frac{1}{\alpha}}\le Cr^\frac{1}{\alpha},
\]
where $k$ is determined by $\alpha$. The  conclusion follows. 

  In the range $2\le p < \infty$, the second of the Clarkson inequalities implies  
\begin{align*}
 \int_{\{\alpha<  u<\beta\} }|D (u-v)/2|^p&\le (M+2pC_0)-(M-2pC_0) \\
 &\le 4pC_0  \le C M^\delta+C,
  \end{align*}
thus concluding the proof.
  \end{proof}
  
\begin{lem}\label{growth est2} 
Let $(X, d, \mu)$ be an unbounded, complete $Q$-Ahlfors regular measure space supporting a $Q$-Poincar\'e inequality for $Q >1$. Suppose $u, v$ are two global Green functions with pole at $x_0\in X$. 
Then there exists a constant $C$ 
such that 
 $$\int_{B(x_0, 1/r)\setminus B(x_0, r)} |D (u-v)|^Q \le C (\log(1/r))^\delta+C, $$
 for all $r\in (0,1]$,
 where $\delta\in [0,1)$.
\end{lem}

\begin{proof}
This follows from the previous lemma in combination with the growth estimates from Corollary \ref{cor-growth},
$$\frac{1}{C} \log(1/d(x,x_0))-C\le u(x) \le C \log(1/d(x,x_0))+C. \qedhere$$
\end{proof}

We recall that a pathwise connected space $X$ is Loewner if \eqref{eq:cap lower bound} holds for all disjoint nondegenerate continua in $X$. The estimate in Lemma \ref{lem:cap lower bound} from \cite{HK98} implies that PI-spaces are Loewner. 
\begin{lem}Suppose $E$ and $F$ are nondegenerate, disjoint continua in a Loewner space $X$ connecting
$B(x_0,r/2)$ to $X\setminus B(x_0,r)$, and suppose that $u$ is a continuous Sobolev function such that $u|_E\ge 1$ and $u|_F\le 0$. 
Then for some $L\ge 2$ and $C>0$
 (independent of $u$ and $r$) we have 
$$ \int_{B(x_0, Lr)\setminus B(x_0, r/L)} |Du|^Q\ge C.$$ 
\end{lem}

\begin{proof} The Loewner property (Lemma \ref{lem:cap lower bound}) implies there exists a constant $c_0>0$ such
that 
\begin{equation}\label{LWNR}
\int_{X} |Df|^Q\ge c_0\end{equation}
for every Sobolev function with $f(x) \ge 1$ for $x\in E$ and $r/2\le|x-x_0|\le r$
and $f(x)\le 0$ for $x\in F$ and  $r/2\le|x-x_0|\le r$.

 Let $v=\min\{ u_+,1\}$. Then $0\le v\le 1$ and  $|Dv|\le |Du|$.  
 
 Fix a constant $M>0$ such that $(1/M)^{1/Q}<1/2$. If we choose  $L\ge 1$ large enough (independent of $u$ and $r$), then there exists  a Lipschitz function $0\le \eta\le1 $ 
 such that $\eta(x)=1$ for $x\in X$ with  $r/2\le |x-x_0|\le r$, $\eta(x)=0$ for 
 $x\in X$ with $x\in B(x,r/L)$ or $x\in X\setminus B(x_0,Lr)$, and 
 $$\int_{X} |D\eta|^Q\le c_0/M. $$
 
  Now choosing $f=v\eta$ in the \eqref{LWNR}, one has 
   $Df=v d\eta + \eta Dv$, and consequently  
  \begin{align*}
  c_0^{1/Q}&\le \biggl(\int |Df|^Q\biggr)^{1/Q}\le  \biggl(\int |vD\eta|^Q\biggr)^{1/Q}+
   \biggl(\int |\eta Dv|^Q\biggr)^{1/Q}\\
   &\le \biggl(\int |D\eta|^Q\biggr)^{1/Q}+  \biggl(\int_{B(x_0, Lr)\setminus B(x_0, r/L)} |Dv|^Q\biggr)^{1/Q}  \\
   &\le \frac12   c_0^{1/Q}+  \biggl(\int_{B(x_0, Lr)\setminus B(x_0, r/L)} |Du|^Q\biggr)^{1/Q}, 
   \end{align*}
  and the statement follows.
  \end{proof}

\begin{lem} \label{lem:lowcapbd} Suppose $E$ and $F$ are nondegenerate, disjoint continua in a Loewner space $X$ connecting
$B(x_0,r_1)$ to $X\setminus B(x_0,r_2)$, where $0<r_1<r_2$, and suppose that $u$ is a continuous Sobolev function such that $u|_E\ge 1$ and $u|_F\le 0$. 
Then there exists a constant (only depending on the space) such that
$$ \int_{B(x_0, r_2)\setminus B(x_0, r_1)} |Du|^Q\ge C \log (r_2/r_1)-C.$$  
\end{lem}
\begin{proof} This follows from the previous lemma, since we can place
about $\log (r_2/r_1)$ pairwise disjoint annuli of the form 
$B(x_0, Lr)\setminus B(x_0, r/L)$ inside the annulus $B(x_0, r_2)\setminus B(x_0, r_1)$.
\end{proof}

\begin{lem}\label{limits exist}  Let $(X, d, \mu)$ be an unbounded, complete $Q$-Ahlfors regular measure space supporting a $Q$-Poincar\'e inequality for $Q >1$. Let $x_0\in X$. Let $h$ be a bounded and continuous Sobolev function in $X\setminus\{x_0\}$. Suppose that for each $\alpha\in \R$ the sets $\{h>\alpha\}$ and  $\{h<\alpha\}$ 
have no (non-empty) components that are compactly contained in $X\setminus\{x_0\}$ and that there are constants $\delta\in [0,1)$ and $C>0$ such that 
 \begin{equation}\label{sublinear growth}
 \int_{B(x_0, 1/r)\setminus B(x_0, r)} |D h|^Q \le C (\log(1/r))^\delta+C, 
 \end{equation}
 for all $r\in (0,1]$.
 Then the limits $\lim_{d(x, x_0)\to 0} h(x)$ and $\lim_{d(x, x_0)\to \infty} h(x)$  exist.
\end{lem}

\begin{proof} In order to show that $\lim_{x\to x_0} h(x)$ exists, we argue by
 contradiction and assume that 
 $$\liminf_{x\to x_0}  h(x)< \limsup_{x\to x_0}  h(x). $$
Then we can choose real numbers $a,b\in \R$ such that 
 $$ \liminf_{x\to x_0}  h(x_0)<a<b<  \limsup_{x\to x_0}  h(x_0).$$
It follows that the sets $A\coloneqq \{h<a\}$ and $B\coloneqq \{h>b\}$ contain the point $x_0$ in their closures. 
So we can find sequences $\{y_n\}$ and $\{z_n\}$ in $A$ and $B$, respectively, such that $y_n\to x_0$ and $z_n\to x_0$ as $n\to \infty$.

Let $E_n$ be the closure of the component of $A$ containing $y_n$.
Since the open set $A$ has no components compactly contained in 
$X\setminus \{x_0\}$, we have  $x_0\in E_n$ or $E_n$ is unbounded. Similarly, if $F_n$ is the closure of the component of $B$ containing $z_n$, 
then   $x_0\in F_n$ or  $F_n$ is unbounded.    

We now consider $u=(h-a)/(b-a)$. Then $u\ge 1$ on $F_n$ and $u\le 
0$ on $E_n$ for each $n\in \N$. Moreover, $Du=Dh/(b-a)$ and so for some 
$C'>0$ we have
 $$\int_{B(x_0, 1/r)\setminus B(x_0, r)} |D u|^Q \le C' (\log(1/r))^\delta+C'$$
 for all $r\in (0,1]$. 
 
{\em Claim.} There exist $R>0$ such that for all small enough $r>0$ there exist 
$n,k\in \N$ such that $E=E_n$ and $F=F_k$ connect $B(x_0,r)$ and $X\setminus 
B(x_0, R)$.
 
 \smallskip
 To see this, we consider four cases.
 
 \smallskip
 {\em Case 1:}  $x_0\in E_n$ for some $n\in \N$ and $x_0\in F_k$ for some 
 $k\in \N$.
 
 \smallskip
 If $R\coloneqq \min\{d(y_n, x_0), d(z_n,x_0)\}$ and $r<R$, then 
 $E=E_n$ and $F=F_k$ connect $B(x_0,r)$ and $X\setminus B(x_0, R)$, because 
  the latter set contains $y_n$ and $z_k$.
 
  \smallskip
 {\em Case 2:}  $x_0\in E_n$ for some $n\in \N$ and $x_0\in F_k$ for no 
 $k\in \N$.
 
   \smallskip
 Then all $F_k$ must be  unbounded. Let $R=d(x_0, y_n)$. If $r<R$, then we can choose $k\in \N$ such that $d(x_0, z_n)<r$. Then $E=E_n$ and $F=F_k$ connect
  $B(x_0,r)$ and $X\setminus B(x_0, R)$. 
  
  \smallskip
 {\em Case 3:}  $x_0\in E_n$ for no $n\in \N$ and $x_0\in F_k$ for some 
 $k\in \N$.
 
   \smallskip
   This is similar to Case 2 with the roles of the sets $E_n$ and $F_k$ reversed.
   
   \smallskip
 {\em Case 4:}  $x_0\in E_n$ for no $n\in \N$ and $x_0\in F_k$ for no
 $k\in \N$.
 
   \smallskip
  Consequently all sets $E_n$ and $F_k$ are unbounded. Let $R=1$ and $0<r<1$. 
  Then we can find $n,k\in \N$ such that $y_n,z_k\in B(x_0,r)$. Then 
  $E=E_n$ and $F=F_k$ connect
  $B(x_0,r)$ and $X\setminus B(x_0, R)$. 
  
    \smallskip
    Since these four cases exhaust all possibilities, then the Claim follows.
    The Claim in combination with   Lemma~\ref{lem:lowcapbd} 
    (applied to $r_1=r$ and $r_2=R$) shows 
  that with some constants $c_1, c_2>0$  for all small enough $r>0$ we have  
 \begin{align*} c_1\log(R/r)-c_2&\le \int_{B(x_0, R)\setminus B(x_0,r)}|Dh|^Q\\
 &\le 
  \int_{B(x_0, 1/r)\setminus B(x_0,r)}|Dh|^Q\le C(\log(1/r))^\delta+C.
  \end{align*}
  As $r\to 0$, the first term in this inequality grow linearly with $\log(1/r)$,
  while the last term grow sublinearly as $\delta\in [0,1)$. This is impossible
  and we arrive at a contradiction.
  
  This shows that the limit  $\lim_{x\to x_0} h(x)$ exists, as desired.
  The existence of the limit  $\lim_{x\to \infty} h(x)$ is shown along similar lines.
  Again we argue by contradiction, and find sets $E_n$ and $F_n$ in a similar 
  way as before. In this case  one can show that there exists $R>0$ that that 
  for all small enough $r>0$ there exist $n,k\in \N$ such that 
  $E_n$ and $F_k$ connects $B(x_0, R)$ and $B(x_0, 1/r)$. Based on Lemma~\ref{lem:lowcapbd} in combination with our assumptions, one obtains a contradiction 
  as before and concludes  that  $\lim_{x\to \infty} h(x)$ indeed exists.
\end{proof}


We finally complete the proof of Theorem \ref{thm2}.

\begin{proof}
Let $h=u-v$ in the previous lemma, Lemma \ref{global bound} and Lemma \ref{growth est2} implies $h$ is bounded and satisfies sublinear growth \eqref{sublinear growth}. Comparison principle implies that for each $\alpha\in \R$ the sets $\{h>\alpha\}$ and  $\{h<\alpha\}$ 
have no (non-empty) components that are compactly contained in $X\setminus\{x_0\}$. Hence, we obtain that $\displaystyle\lim_{d(x,x_0)\to 0}(u(x)-v(x))$ and $\displaystyle\lim_{d(x,x_0)\to \infty}(u(x)-v(x))$  both exist from the previous lemma. 
Setting $$\alpha=\lim_{d(x,x_0)\to 0}(u(x)-v(x))\quad\quad {\rm and} \quad\quad\beta=\lim_{d(x,x_0)\to \infty}(u(x)-v(x))$$
we will prove that $\alpha=\beta$. Without loss of generality, we assume $\beta\ge \alpha$. Let $\varepsilon>0$ be fixed. There exists $r>0$ such that for all $x\in B_r(x_0)$, 
\begin{equation}\label{limit0}
-\varepsilon<u(x)-v(x)-\alpha<\varepsilon,
\end{equation}
and there exists $R>0$ such that for all $x\in X\setminus B_R(x_0)$, 
\begin{equation}\label{limitinf}
-\varepsilon<u(x)-v(x)-\beta<\varepsilon.
\end{equation}


Since $u(x),v(x)\to \infty$ as $x\to x_0$, we can choose $M$ sufficiently large such that 
\[
\{ u\ge M\}\subset B_r(x_0)\quad\quad\quad\quad\text{and} \quad\quad\quad\{ v\ge M-\alpha-\varepsilon\}\subset B_r(x_0).
\]
And \eqref{limit0} implies that
\[
\{ u\ge M\}\subset \{ v\ge M-\alpha-\varepsilon\},
\]
or equivalently, 
\[
\{ v< M-\alpha-\varepsilon\}\subset \{ u< M\}.
\]

On the other end, since $u(x),v(x)\to -\infty$ as $x\to\infty$, we can choose $m$ sufficiently small such that 
\[
\{ u\le m\}\subset X\setminus B_R(x_0)\quad\quad\quad\quad\text{and} \quad\quad\quad\{ v\le m-\beta+\varepsilon\}\subset X\setminus B_R(x_0).
\]
And \eqref{limitinf} implies that
\[
\{ u\le m\}\subset \{ v\le m-\beta+\varepsilon\},
\]
or equivalently, 
\[
\{v>m-\beta+\varepsilon\}\subset \{ u> m\}.
\]
Thus, it follows that
\[
\{ m-\beta+\varepsilon<v< M-\alpha-\varepsilon\} \subset \{ m<u<M\}.
\]
In view of the renormalization in terms of capacity  of $u,v$ as Green functions, one obtains
\[
\begin{aligned}
(M-\alpha-\varepsilon-m+\beta-\varepsilon)^{1-Q}&=
\capa_Q( \{v> m-\beta+\varepsilon\}, \{v> M-\alpha-\varepsilon\})\\
&\ge \capa_Q(\{u> m\},\{ u>M\})\\
&=(M-m)^{1-Q}.
\end{aligned}
\]
Hence, $0\le \beta-\alpha\le 2\varepsilon$. Let $\varepsilon\to 0$, we get that $\alpha=\beta$.

Finally, we show $u(x)=v(x)$ for all $x\in X\setminus\{x_0\}$. Since 
\[
\lim_{d(x,x_0)\to 0}(u(x)-v(x))=\lim_{d(x,x_0)\to \infty}(u(x)-v(x))=\alpha,
\]
for all $\varepsilon$, there exists a small $r$ and a large $R$ such that for all $x\in \partial({B_R(x_0)\setminus \overline{B_r(x_0)}})$
\[
-\varepsilon\le u(x)-v(x)-\alpha\le \varepsilon.
\]
Comparison Principle implies that the above inequality holds for all $x\in B_R(x_0)\setminus \overline{B_r(x_0)}$. Let $\varepsilon\to 0$, we get that $u(x)=v(x)+\alpha$ for all $x\in X\setminus\{x_0\}$. Since $u(y_0)=v(y_0)$, it follows that $\alpha=0$ and so
$u(x)=v(x)$
for all $x\in X\setminus\{x_0\}$, concluding the proof.
\end{proof}

\setcitestyle{square}

\section{Proof of Theorem \ref{fundamental solution}}


We first recall a Theorem on the local behavior of a Green function near a singularity from \cite[Theorem 5.2]{DGM10}. For related results with less restrictive hypotheses, see \cite{BBLe23}.

\begin{thm}\label{DGM5.2}
Let $(X, d, \mu)$ be a complete Ahlfors $Q$-regular metric measure space that supports a $Q$-Poincar\'e inequality for $Q>1$.  Let $\Omega$ be a relatively compact domain in $X$ and $x_0\in \Omega$. If $u$ is a $Q$-harmonic Green function in $\Omega$ with singularity at $x_0$, then there exist positive constants $C,R_0, R_1$ such that $R_1\le \frac{R_0}{2}$ and for any $0<r<R_1$ and $x\in B(x_0,r)$ we have
\[
C^{-1}\log\biggl(\frac{R_0}{d(x,x_0)}\biggr)\le u(x)\le C\log\biggl(\frac{R_0}{d(x,x_0)}\biggr).
\]

\end{thm}


\begin{lem}\label{corollary 2 in Serrin 65}
Let $(X, d, \mu)$ be a complete Ahlfors $Q$-regular metric measure space that supports a $Q$-Poincar\'e inequality for $Q>1$. Let $\Omega$ be a relatively compact domain in $X$ and let $u$ be a $Q$-harmonic Green function in $\Omega$ with singularity at $x_0$. Then there exists a constant $R_0$ such that $B(x_0, R_0)\subset \Omega$ and  for $0<r<\frac{R_0}{4}$ the function $\theta=|D u|^{Q-2} Du$ is integrable in $B(x_0,r)\setminus \{x_0\}$ and 
$$\int_{B(x_0,r)\setminus\{x_0\}} |\theta|\,    \le C r\log \biggl(\frac{2R_0}{r}\biggr)^{Q-1}.
$$
\end{lem}
\begin{proof} We first show a Caccioppoli-type estimate for $u$: there exists a constant $C$ such that on balls $B(x_0, r)\subset\subset B(x_0, R)\subset\subset \Omega$ with $R\ge2r$, 
\begin{equation}\label{Caccioppoli}
\int_{B_r\setminus \{x_0\}} |Du|^Q\,  \le Cr^{-Q}\int_{B_R\setminus B_r} |u|^Q\,  .
\end{equation}
Define a cutoff function $\eta\in N_0^{1,Q}(B_R)$ such that $\eta$ is Lipschitz, $0\le \eta\le1$, $\eta=1$ in $B_r$ and $|D\eta|\le \frac{2}{R-r}$.
Since $u$ is Cheeger $Q$-harmonic in $B(x_0, R)\setminus \{x_0\}$, $u$ and $\phi=u\eta^Q\in N_0^{1,Q}(B_R)$ satisfy the equation \eqref{cpl1}, that is,

\[
\int_{B_R\setminus \{x_0\}} |Du|^{Q-2}Du\cdot D\phi =0,
\]
where
\[
D\phi= \eta^Q Du+p u \eta^{Q-1} D\eta.
\]
This implies that
\[
\begin{aligned}
\int_{B_R\setminus \{x_0\}}  |Du|^{Q}\eta^p&\le p\int_{B_R\setminus \{x_0\}} |Du|^{Q-1}\eta^{Q-1}|D\eta||u|\\
&\le p\Big(\int_{B_R\setminus \{x_0\}} |Du|^{Q}\eta^{Q}\Big)^{\frac{Q-1}{Q}} \Big(\int_{B_R\setminus \{x_0\}} |D\eta|^Q|u|^Q\Big)^{\frac{1}{Q}}.\\
\end{aligned}
\]
Hence, 
\[
\int_{B_R\setminus \{x_0\}}  |Du|^{Q}\eta^p\le p^p\int_{B_R\setminus \{x_0\}} |D\eta|^Q|u|^Q\le \frac{C}{(R-r)^p}\int_{B_R\setminus B_r} |u|^Q.
\]
Note $\eta=1$ on $B_r$ and the Caccioppoli-type estimate \eqref{Caccioppoli} follows immediately. 

Theorem \ref{DGM5.2} implies that there exist constants $C, R_0, R_1$ such that  $0<R_1<\frac{R_0}{2}$ and
\begin{equation}\label{p=Q}
C^{-1}\log\bigg(\frac{R_0}{d(x,x_0)}\bigg)\le u(x)\le C\log\bigg(\frac{R_0}{d(x,x_0)}\bigg)
\end{equation}
for $0<d(x,x_0)<R_1$. Hence, for $r$ satisfying $0<2r<R_1$ we have
\[
\begin{aligned}
\int_{B_{2r}\setminus B_{r}} |u|^{Q}&\le \sup_{B_{2r}\setminus B_{r}} |u|^{Q} \big( \mu(B_{2r})-\mu(B_{r})\big)\\
&\le C\log\biggl(\frac{R_0}{r}\biggr)^{Q} r^{Q}.
\end{aligned}
\]
Combined with the Caccioppoli-type estimate \eqref{Caccioppoli}, we get that
\[
\int_{B_r\setminus \{x_0\}} |Du|^{Q}\  \le Cr^{-Q} \int_{B_{2r}\setminus B_r} |u|^{Q}\le  C\log\biggl(\frac{R_0}{r}\biggr)^{Q} .
\]
Apply H\"older's inequality, we get that
\[
\begin{aligned}
\int_{B_r\setminus \{x_0\}} |Du|^{Q-1}\  &\le \biggl(\int_{B_r\setminus \{x_0\}} |Du|^{Q}\  \biggr)^{\frac{Q-1}{Q}} \mu(B_r)^{\frac{1}{Q}}\\
&\le  Cr\log\biggl(\frac{R_0}{r}\biggr)^{Q-1}.
\end{aligned}
\]
\end{proof}
\begin{lem}\label{constant}
Let $(X, d, \mu)$ be a complete Ahlfors $Q$-regular metric measure space that supports a $Q$-Poincar\'e inequality for $Q>1$. Let $\Omega$ be a relatively compact domain in $X$ and let $B(x_0,R)\subset \Omega$. Let $\phi \in N^{1,Q}_0(B(x_0,R))$ such that $\phi=1$ in a neighborhood of $x_0$. If $u$ is a $Q$-harmonic Green function in $\Omega$ with singularity at $x_0$ then
there exists $K\in \R$ 
such that
$$\int_{\Omega} |Du|^{Q-2}Du\cdot D\phi  =K.$$
\end{lem}

\begin{proof}
Let $\phi_1, \phi_2\in N^{1,Q}_0(B(x_0,R))$ be such that $\phi_1=1$ and $\phi_2=1$ in a neighborhood of $x_0$. Then $\phi_1-\phi_2\in N^{1,Q}_0(B(x_0,R)\setminus\{x_0\})$. This implies that
\[
\int_{\Omega} |Du|^{Q-2}Du\cdot D\phi =\int_{B(x_0,R)} |Du|^{Q-2}Du\cdot D\phi_1 =\int_{B(x_0,R)} |Du|^{Q-2}Du\cdot D\phi_2.
\]
The proof is complete.
\end{proof}

Finally, we complete the proof of Theorem \ref{fundamental solution}.
\begin{proof}
Let $\phi\in N^{1,Q}_0(B(x_0, R)$ for some $B(x_0, R)\subset \Omega$ be a Lipschitz continuous function and $\phi(x_0)\neq 0$. Let $0\le \eta\le 1$ be a Lipschitz function with $\eta=1$ in $B(x_0, \frac{R}{2})$, $\eta=0$ in $X\setminus B(x_0, R)$ and $|D\eta|\le C/R$ for some constant $C>0$. The function
\[
\psi=(1-\eta) \phi+\eta\phi(x_0)
\]
has compact support in $B(x_0, R)$ and equals $\phi(x_0)$ identically in a neighborhood of $x_0$. By Lemma \ref{constant}, 
$$\int_{\Omega} |Du|^{Q-2}Du\cdot D\psi  =K\phi(x_0).$$
This implies that
\[
\int_{\Omega} |Du|^{Q-2}Du\cdot D\phi\   -\int_{\Omega} |Du|^{Q-2}(Du\cdot D\phi) \eta\   +\int_{\Omega} |Du|^{Q-2}Du\cdot (\phi-\phi(x_0))D\eta\   =K\phi(x_0).
\]
Since $ |Du|^{Q-2}Du $ is integrable and $D\phi$ is bounded, we have $$ \int_{\Omega} |Du|^{Q-2}(Du\cdot D\phi) \eta= \int_{B(x_0, R)} |Du|^{Q-2}(Du\cdot D\phi) \eta$$ converges to $ 0$ as $R\to 0$. On the other hand, we have
\[
\begin{aligned}
\left|\int_{\Omega} |Du|^{Q-2}Du\cdot (\phi-\phi(x_0))D\eta\right|   &\le \sup_{B_R} |\phi-\phi(x_0)|\int_{B(x_0,R)}|Du|^{Q-1}\cdot |D\eta|  \\
&\le C\int_{B(x_0,R)\setminus\{x_0\}} |Du|^{Q-1}  ,
 \end{aligned}
\]
which tends to zero as $R\to 0$ by Lemma \ref{corollary 2 in Serrin 65}. Combining the above estimates, we obtain 
\[
\int_{\Omega} |Du|^{Q-2}Du\cdot D\phi\   =K\phi(x_0),
\]
concluding the proof.
\end{proof}

\end{document}